\documentclass[12pt]{amsart}
\usepackage[utf8]{inputenc}
\usepackage[T1]{fontenc}
\usepackage{url}
\usepackage{color}
\usepackage{hyperref}
\usepackage{enumerate}
\usepackage{array}
\usepackage{nicefrac}
\usepackage[section]{placeins}
\usepackage{amsmath,amssymb,enumerate,amsthm}
\usepackage{latexsym}
\usepackage{setspace}
\usepackage{mathtools}
\usepackage{longtable}
\usepackage{mathdots}
\usepackage{graphicx, subfig}
\usepackage{anysize}
\usepackage[linguistics,edges]{forest}

\usepackage{enumitem}

\theoremstyle{plain}

\newtheorem{theorem}{Theorem}[section]
\newtheorem{lemma}[theorem]{Lemma}
\newtheorem{prop}[theorem]{Proposition}
\newtheorem{coro}[theorem]{Corollary}

\theoremstyle{definition}
\newtheorem{definition}[theorem]{Definition}

\newtheorem{example}[theorem]{Example}

\theoremstyle{remark}
\newtheorem{remark}[theorem]{Remark}

\numberwithin{equation}{section}

\def\H{\mathbb{H}}

\def\F{\mathcal{F}}
\def\N{\mathbb{N}}
\def\C{\mathbb{C}}
\def\R{\mathbb{R}}
\def\Z{\mathbb{Z}}

\def\P{\mathbb{P}}

\def\M{\mathcal{M}}

\def\SL{\mathrm{SL}(2,\mathbb{Z})}
\def\GL{\mathrm{GL}(2,\mathbb{Z})}

\def\Im{\mathrm{Im}}
\def\Re{\mathrm{Re}}

\def\val{\mathrm{val}}
\def\op{\mathrm{op}}

\newcommand{\bmatch}{\textit{b-match}}
\newcommand{\fmatch}{\textit{f-match}}

\def\Far{\F_{\frac{0}{1},\frac{1}{0}}}
\def\Fartwo{\F_{\frac{0}{1},\frac{1}{2}}}
\def\SLN{\mathrm{SL}(2,\N_0)}
\def\GLN{\mathrm{GL}(2,\N_0)}

\usepackage{tikz-cd}

\begin{document}
\title{A Lyapunov exponent attached to modular functions}

\author{P.~Bengoechea}
\address{Universitat de Barcelona, Departament de Matem\`atiques i Inform\`atica, Gran Via de les Corts Catalanes, 585, L'Eixample, 08007 Barcelona, Spain}
\email{bengoechea@ub.edu}
\thanks{P.~Bengoechea's research is supported by Ram\'on y Cajal grant RYC2020-028959-I}

\author{S.~Herrero}
\address{Universidad de Santiago de Chile, Dept.~de Matem\'atica y Ciencia de la Computaci\'on, Av.~Libertador Bernardo O'Higgins 3363, Santiago, Chile, and ETH, Mathematics Dept., CH-8092, Z\"urich, Switzerland}
\email{sebastian.herrero.m@gmail.com}
\thanks{S.~Herrero's research is supported by ANID/CONICYT FONDECYT Iniciaci\'on grant 11220567 and by SNF grant CRSK-2{\_}220746}

\author{\"O.~Imamo\={g}lu}
\address{ETH, Mathematics Dept., CH-8092, Z\"urich, Switzerland}
\email{ozlem@math.ethz.ch}
\thanks{\"O.~Imamo\=glu's research  is supported by SNF grant 200021-185014}

\begin{abstract}
   To each weakly holomorphic modular function $f\not \equiv 0$  for $\SL$, which is non-negative on the geodesic arc $\{e^{it} : \pi/3\leq t\leq 2\pi/3\}$,   we attach a $\GL$-invariant map $\Lambda_f:\P^1(\R)\to \R$ that generalizes the Lyapunov exponent function introduced by Spalding and Veselov.  We prove that it takes every value between $0$ and $\Lambda_f\left(\frac{1+\sqrt{5}}{2}\right)$ and it gives an increasing  convex function  on the Markov irrationalities when ordered using their parametrization by Farey fractions in $[0,1/2]$. In the case of quadratic irrationals $w$ with purely periodic continued fraction expansion, the value $\Lambda_f(w)$ equals the real part of the cycle integral of $f$ along the associated geodesic $C_w$ on the modular surface, normalized with the word length of the associated hyperbolic matrix $A_w$ as a word in the generators $T=\left(\begin{smallmatrix}
    1 & 1 \\ 0 & 1
\end{smallmatrix}\right)$ and $V=\left(\begin{smallmatrix}
    1 & 0 \\ 1 & 1
\end{smallmatrix}\right)$. These results are related to conjectures of Kaneko who observed several similar behavior for the cycle integrals of the modular $j$ function when normalized by the hyperbolic length of the geodesic $C_w$. 
\end{abstract}

\maketitle


\section{Introduction}

 In the papers \cite{SV17,SV18}, Spalding and Veselov defined and studied a function
 \[\Lambda: \P^1(\R) \to\R\]
 which they called the \emph{Lyapunov exponent} and showed, among other properties, the following:
\begin{itemize}[leftmargin=*]
\item For  almost all $x\in \R$, including all rational numbers, $\Lambda(x)=0$.
    \item For all $x\in \P^1(\R)$ and all $\left(\begin{smallmatrix}
        a & b \\ c & d
    \end{smallmatrix}\right)\in \GL$, $\Lambda\left(\frac{ax+b}{cx+d}\right)=\Lambda(x)$. 
    \item $\Lambda(\mathbb{P}^1(\mathbb{R}))=[0,\Lambda(\phi)]$ where $\phi=\frac{1+\sqrt{5}}{2}$ is the golden ratio and $\Lambda(\phi)=\log(\phi)=0.481\ldots$.
    \item If one denotes by $\M$ the set of Markov irrationalities, namely the set of real quadratic irrationalities constructed from $\phi=[\overline{1;1}]$ and the silver ratio $\psi=1+\sqrt{2}=[\overline{2;2}]$ by iterative conjunction of continued fraction expansions  (see Figure \ref{fig:Markov_tree}), then the restriction of $\Lambda$ to $\M$ satisfies $\Lambda(\M)\subset [\Lambda(\psi),\Lambda(\phi)]$ with $\Lambda(\psi)=\frac{\log(\psi)}{2}=0.440\ldots$.
\item When the Markov irrationalities $w\in \M$ are parametrized by the Farey fractions in $[0,1/2]$, the resulting map $p/q\mapsto \Lambda\left(w\left(p/q\right)\right)$ extends to a continuous, decreasing and convex function $\tilde{\Lambda}:[0,1/2]\to \R$.
\end{itemize}

In a completely different direction,  Kaneko defined in \cite{Kan09} a function $\val(w)$ for the ``values'' of the classical modular invariant $j$ on real quadratic irrationalities $w$. The real part of Kaneko's $\val$ function satisfies the same $\GL$-invariance as $\Lambda$ above, and he observed several Diophantine properties of this function. Among others, he conjectured:
\begin{itemize}[leftmargin=*]
\item For all quadratic irrationalities $w$, $\Re(\val(w))\in [\val(\phi),744]$ where $\val(\phi)=706.3248\ldots$.
\item When restricted to the Markov irrationalities, the real part of the $\val$ function satisfies  $\Re(\val(\M))\subset [\val(\phi),\val(\psi)]$ where $\val(\psi)=709.8928\ldots$.
\item The restriction of $\Re(\val(w))$ to $w\in \M$ is an increasing function (in Kaneko's original paper \cite{Kan09}, this  is stated in terms of an ``interlacing property'' on the Markov tree).
\end{itemize}
Some of Kaneko's conjectures were proved in \cite{BI19,BI20}.

If $w$ is a quadratic irrationality with purely periodic continued fraction expansion $w=[\overline{a_1;a_2,\ldots,a_\ell}]$, where we assume $n\geq 2$ even by doubling the period of $w$ is necessary, the value $\Lambda(w)$ of Spalding and Veselov's  function is
$$\Lambda(w)=\frac{\log(\varepsilon_w)}{a_1+\ldots+a_\ell},$$
where $2\log(\varepsilon_w)=\int_{C_w}1ds$ is the hyperbolic length of a closed geodesic $C_w$ on the modular surface\footnote{As usual $\H:=\{\tau \in \C:\Im(\tau)>0\}$ and $\SL$ acts on $\H$ by fractional linear transformations.} $\SL\backslash \H$ naturally associated to $w$, while the value of Kaneko's $\val$ function is 
$$\val(w)=\frac{\int_{C_w}jds}{2\log(\varepsilon_w)}.$$
Motivated by the results of Spalding and Veselov, as well as the conjectures of Kaneko, we define a Lyapunov exponent $\Lambda_f$ associated to any weakly holomorphic modular function $f$ for $\SL$ which is  real and non-negative on the geodesic arc $\{e^{it}:t \in [\pi/3,2\pi/3]\}$. In the case of $f=j$ and for a quadratic irrationality $w$ as above, our function is
\[\Lambda_j(w)=\frac{\Re\left(\int_{C_w}jds\right)}{a_1+\ldots+a_\ell}.\]
Noting that $a_1+\cdots +a_\ell$ is also the word length in $T:=\left(\begin{smallmatrix}
    1 & 1 \\ 0 & 1
\end{smallmatrix}\right)$ and $V:=\left(\begin{smallmatrix}
    1 & 0 \\ 1 & 1
\end{smallmatrix}\right)$ of the hyperbolic matrix $A_w=T^{a_1}V^{a_2}\cdots T^{a_{\ell-1}}V^{a_\ell}$ fixing $w$, the function $\Lambda_j$ can be interpreted as an analog of the function $\Re(\val(w))$ but normalized with the word length of the matrix $A_w$  instead of the length of geodesic $C_w$. 

In order to define $\Lambda_f(x)$ for a real number $x$, which is the main object of study of this paper, we introduce some definitions and notation. Given a hyperbolic matrix $A\in \SL$, let $\tilde{w},w\in \R$ denote the attracting and repelling  fixed points of $A$, respectively. Let $Q(x,y)=ax^2+bxy+cy^2$ be an integral, primitive, indefinite quadratic form  satisfying $Q(w,1)=Q(\tilde{w},1)=0$. Among the two quadratic forms $Q$ and $-Q$ we choose the one satisfying $\mathrm{sgn}(a)=\mathrm{sgn}(w-\tilde{w})$. With this convention, we have the formulas
\begin{equation*}
    w=\frac{-b+\sqrt{D}}{2a},\quad \tilde{w}=\frac{-b-\sqrt{D}}{2a},
\end{equation*}
where $D:=b^2-4ac>0$ is the discriminant of $Q$. Then, given a weakly holomorphic modular function $f$  for $\SL$ we put
\begin{equation}\label{eq:def_cycle_integral_intro}
    I_f(A):=\int_{\tau_0}^{A\tau_0}\frac{f(\tau)\sqrt{D}}{Q(\tau,1)}d\tau \quad (\text{any } \tau_0\in \H).
\end{equation}

To define $\Lambda_f(x)$ for $x\in (0,\infty)$ we will use the rational approximations of $x$ given by the Farey tree    with starting Farey fractions $\frac{0}{1}$ and $\frac{1}{0}$. More precisely, consider the infinite tree $\Far$ generated from $\frac{0}{1}$ and $\frac{1}{0}$ by taking \emph{Farey medians}
\[\frac{a}{b}\oplus \frac{c}{d}:=\frac{a+c}{b+d}\]
between pairs of Farey fractions $\frac{a}{b}<\frac{c}{d}$ as shown in Figure \ref{fig:Far}. We call $\frac{a}{b}\oplus \frac{c}{d}$ a \emph{direct descendant} of both $\frac{a}{b}$ and $\frac{c}{d}$, and call $\frac{a}{b}$ (resp. $\frac{c}{d}$) the \emph{left} (resp. \emph{right}) \emph{parent} of $\frac{a}{b}\oplus \frac{c}{d}$.
\begin{figure}[h!]
\centering
\begin{forest}
[,phantom [$\frac{0}{1}$,name=p1] [] [] [] [] [] [] 
[
[$\frac{1}{1}$,name=p2, no edge,tikz={\draw (p2.north)--(p1.south);}
[
$\frac{1}{2}$ 
[
$\frac{1}{3}$
[$\ldots$]
[$\ldots$]]
[
$\frac{2}{3}$
[$\ldots$]
[$\ldots$]]
]
[
$\frac{2}{1}$
[
$\frac{3}{2}$
[$\ldots$]
[$\ldots$]] 
[
$\frac{3}{1}$
[$\ldots$]
[$\ldots$]]]]]
[] [] [] [] [] []  [$\frac{1}{0}$,name=p3] tikz={\draw (p2.north)--(p3.south);}
]
\end{forest} \caption{Farey tree $\Far$.}\label{fig:Far}
\end{figure}

The sequence of sets 
\[\Far^{(0)} := \left\{ \frac{0}{1}, \frac{1}{0} \right\},   \quad   \Far^{(1)}  :=  \left\{ \frac{0}{1}, \frac{1}{1}, \frac{1}{0} \right\},  \quad   \Far^{(2)} :=  \left\{ \frac{0}{1}, \frac{1}{2}, \frac{1}{1},\frac{2}{1}, \frac{1}{0} \right\}, \quad \ldots\] 
is increasing and their union is $\mathbb{Q}^+\cup \{0,\infty\}$. We call $\Far^{(n)}$ the \emph{$n$-th level of the Farey tree $\Far$}, and Farey fractions that are consecutive in $\Far^{(n)}$ for some $n\geq 0$ are said to be \emph{neighbors}. For instance, $\frac{1}{2}$ and $\frac{1}{1}$ are neighbors whereas $\frac{1}{2}$ and $\frac{2}{1}$ are not. It is well known that two Farey fractions $\frac{a}{b}<\frac{c}{d}$ in $\Far$ are neighbors if and only if $ad-bc=-1$. 

Note that every pair of consecutive Farey fractions $\frac{a}{b},\frac{c}{d}$ in level $n\geq 1$ has either $\frac{a}{b}$ as a direct descendant of  $\frac{c}{d}$ or $\frac{c}{d}$ as a direct descendant of $\frac{a}{b}$. Hence, on the tree, either $\frac{c}{d}$ is above $\frac{a}{b}$ or $\frac{a}{b}$ is above $\frac{c}{d}$, respectively.

It is easy to see that if $\frac{a}{b}<\frac{c}{d}$ are two Farey fractions that are neighbors, then $A:=\left(\begin{smallmatrix}
    c & a \\ d & b
\end{smallmatrix}\right)$ is the unique matrix in $\SLN$ satisfying 
\(A(0)=\frac{a}{b},  A(1)=\frac{a+c}{b+d},  A(\infty)=\frac{c}{d}\).  
In the case where $\frac{c}{d}$ is a direct descendant of $\frac{a}{b}$ we have the diagram in Figure \ref{fig:diagram_A}. 

\begin{figure}[h!]
\centering
\begin{forest}
[{$A(0)=\frac{a}{b}$}
[$\ldots$
 [$\ldots$]
 [$\ldots$]
]
[{$A(\infty)  =  \frac{c}{d}$}
[{$A(1)= \frac{a+c}{b+d}$}]
[$\ldots$]
]
]
\end{forest}
\caption{Image of $\{0,1,\infty\}$ under the matrix $A=\left(\begin{smallmatrix}
    c & a \\ d & b
\end{smallmatrix}\right).$}\label{fig:diagram_A}
\end{figure}
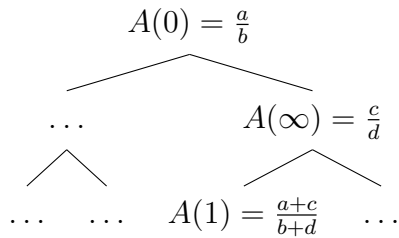

We call $A=\left(\begin{smallmatrix}
    c & a \\ d & b
\end{smallmatrix}\right)$ the \emph{matrix associated to the Farey fraction} $\frac{a+c}{b+d}$. 
We label each edge of the Farey tree below $\frac{1}{1}$ with the matrix associated to the Farey fraction at the bottom of the edge. Figure \ref{fig:labeled_Far} shows the \emph{labeled Farey tree}. Since the monoid $\SLN$ is freely generated by $T$ and $V$, every matrix $A\neq I$ in $\SLN$ appears exactly once as a label in this tree.

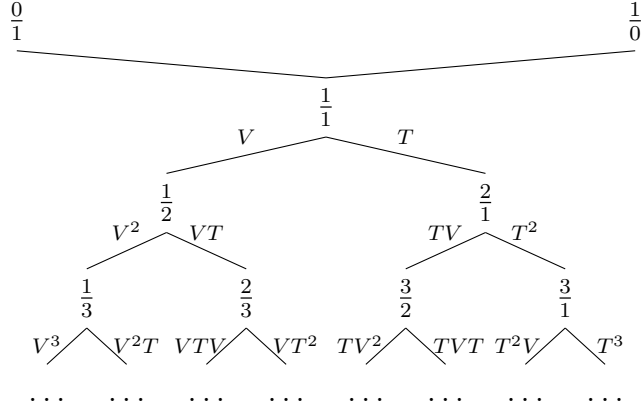
\begin{figure}[h!]
\centering
\begin{forest}
[,phantom [$\frac{0}{1}$,name=p1] [] [] [] [] [] [] 
[
[$\frac{1}{1}$,name=p2, no edge,tikz={\draw (p2.north)--(p1.south);}
[
$\frac{1}{2}$,edge label={node[midway,above]{\tiny{$V$}}}
[
$\frac{1}{3}$,edge label={node[midway,above]{\tiny{$V^2$}}}
[$\ldots$,edge label={node[above]{\tiny{$V^3$}}}]
[$\ldots$,edge label={node[above]{\tiny{\,\,\, $V^2T$}}}]]
[
$\frac{2}{3}$,edge label={node[midway,above]{\tiny{$VT$}}}
[$\ldots$,edge label={node[above]{\tiny{$VTV$\,\,\,}}}]
[$\ldots$,edge label={node[above]{\tiny{\,\,\, $VT^2$}}}]]
]
[
$\frac{2}{1}$,edge label={node[midway,above]{\tiny{$T$}}}
[
$\frac{3}{2}$,edge label={node[midway,above]{\tiny{$TV$}}}
[$\ldots$,edge label={node[above]{\tiny{$TV^2$\,\,\,}}}]
[$\ldots$,edge label={node[above]{\tiny{\,\,\,\,\, $TVT$}}}]] 
[
$\frac{3}{1}$,edge label={node[midway,above]{\tiny{$T^2$}}}
[$\ldots$,edge label={node[above]{\tiny{$T^2V$\, \,}}}]
[$\ldots$,edge label={node[above]{\tiny{\, $T^3$}}}]]]]]
[] [] [] [] [] []  [$\frac{1}{0}$,name=p3] tikz={\draw (p2.north)--(p3.south);}
]
\end{forest}\caption{Labeled Farey tree.}\label{fig:labeled_Far}
\end{figure}

Now, every irrational real number $x\in (0,\infty)$ is the limit of a unique sequence $(x_n)_{n\geq 0}$ of Farey fractions in $\Far$ starting at $x_0=\frac{1}{1}$ and with $x_{n+1}$ a direct descendant of $x_n$ for all $n\geq 0$. The corresponding path in $\Far$ has infinitely many left and right turns moving down the tree, with right turns corresponding to the matrix $V$ and left turns corresponding to $T$ (from the perspective of a viewer standing over the tree and facing south). Similarly, every rational number $\frac{p}{q}\in (0,\infty)$ is the limit of exactly two infinite paths in $\Far$. Indeed, starting at $\frac{1}{1}$   one can reach $\frac{p}{q}$ on the tree after finitely many right and left turns. At $\frac{p}{q}$ we can then make a right turn followed by infinitely many left turns or alternatively make a left turn followed by infinitely many right turns. For the moment, let us choose the unique path with infinitely many right turns.

\begin{example}
    For the rational number $\frac{2}{3}$ the associated sequence of Farey fractions is
    \[x_0=\frac{1}{1}, x_1=\frac{1}{2}, x_2=\frac{2}{3} \text{ and }x_{k+3}=\frac{3+2k}{4+3k}\, \, \text{ for }k\geq 0. \]
\end{example}

\begin{example}
        For the golden ratio $\phi=\frac{1+\sqrt{5}}{2}=1.61803\ldots$ the sequence of Farey fractions is
    \[x_0=\frac{1}{1}, x_1=\frac{2}{1}, x_2=\frac{3}{2} \text{ and }x_{k+3}=x_{k+2}\oplus x_{k+1}\, \, \text{ for }k\geq 0.  \]
\end{example}

To every infinite  path $(x_n)_{n\geq 0}$ in $\Far$ as above, we associate a sequence of matrices $(A_n)_{n\geq 0}$ in $\mathrm{SL}(2,\N)$ as follows: Put $A_0:=I$ and for $n\geq 1$ define
\[A_n:=\left\{\begin{array}{ll}
     A_{n-1}V& \text{if }x_{n}<x_{n-1},\\
    A_{n-1}T & \text{if }x_{n}>x_{n-1}.\\
\end{array} \right.\]
Then, $A_n$ is simply the matrix associated to the Farey fraction $x_n=A_n(1)$.  Observe that the matrices in the sequence $(A_n)_{n\geq 0}$ are hyperbolic except for only finitely many exceptions in the beginning. Hence, given $x$ there exists $n_0\geq 1$ such that for all $n\geq n_0$ the matrix $A_n$ is hyperbolic.

For an irrational number $x\in (0,\infty)$ with continued fraction  expansion \footnote{All the continued fractions appearing in this paper are simple.}
\[x=[a_1;a_2,\ldots ]=a_1+\dfrac{1}{a_2+\dfrac{1}{a_3+\dfrac{1}{\ddots}}},\]
the corresponding sequence $(A_n)_{n\geq 0}$ is given explicitly as follows: 
\begin{equation}\label{eq:hyp_matrix_seq_intro}
   A_n=\left\{
\begin{array}{ll}
T^n & \text{if }0\leq n\leq a_1, \\
    T^{a_1}V^{a_2}\cdots V^{a_i}T^{k} & \text{if $n=a_1+\ldots +a_i+k, 0\leq k\leq a_{i+1}$ and $i\geq 1$ even},  \\
    T^{a_1}V^{a_2}\cdots T^{a_i}V^{k} & \text{if $n=a_1+\ldots +a_i+k, 0\leq k\leq a_{i+1}$ and $i\geq 1$ odd}.
\end{array}
\right. 
\end{equation}

\begin{definition}\label{def:Lambda_f}
    Let $f\not \equiv 0$ be a weakly holomorphic modular function  for $\SL$, which is real and  non-negative on the geodesic arc $\{e^{it} : \pi/3\leq t\leq 2\pi/3\}$. For any real number  $x\in (0,\infty)$  with the associated  sequence of  hyperbolic matrices $(A_n)_{n\geq n_0}$ as above   we define the \emph{$f$-Lyapunov exponent} of $x$ as
\begin{equation}\label{eq:def_lambda_f_intro}
  \Lambda_f(x):=\limsup_{n\to \infty}\frac{\Re(I_f(A_n))}{n}.  
\end{equation}
\end{definition}

In the case that $f\equiv 1$ this is twice the Lyapunov exponent function $\Lambda(x)$ studied by Spalding and Veselov in \cite{SV17,SV18}. Along the paper, whenever we refer to $\Lambda_f$ we implicitly assume that $f$ is as in Definition \ref{def:Lambda_f}. Our first theorem is the following.

\begin{theorem}\label{thm:first_properties_Lambda_f}
  The limit superior in \eqref{eq:def_lambda_f_intro} is finite for all $x\in (0,\infty)$ and $\Lambda_f$ extends to a function $\Lambda_f:\P^1(\R)\to \R$ satisfying
   \[\Lambda_f\left(\frac{ax+b}{cx+d}\right)=\Lambda_f(x) \quad \text{for all }x\in \P^1(\R) \text{ and }\begin{pmatrix}
        a & b \\ c & d
    \end{pmatrix}\in \GL.\]
   Moreover, $\Lambda_f(x)=0$ for all $x\in \P^1(\mathbb{Q})$.
\end{theorem}

\begin{remark}
As in \cite{SV17} one can show that $\Lambda_f(x)=0$ for almost all $x\in \P^1(\R)$ with respect to the Lebesgue measure (see Remark \ref{rmk:vanishing_almost_everywhere} for details).
\end{remark}

Since $\Lambda_f(x)=0$ for $x\in\mathbb Q$ we can restrict ourselves to irrationals  $x\in (0,\infty)$. We first look  at values of the $f$-Lyapunov exponent in the special class of  quadratic irrationalities.  Recall that real quadratic irrationalities are characterized by the property of having (eventually) periodic continued fractions expansions.

\begin{theorem}\label{thm:Lambda_f_on_quad_irrationals}
   Let $x=[a_1;a_2,\ldots,a_r,\overline{b_1,\ldots,b_\ell}]$ be a quadratic irrationality with $\ell\geq 2$ even\footnote{The condition of $\ell$ being even is not restrictive since one can always double the period of a continued fraction expansion.}. 
   Then, we have
\[\Lambda_f(x)=
\frac{\Re(I_f(T^{b_1}V^{b_2}\cdots T^{b_{\ell-1}}V^{b_\ell}))}{b_1+\ldots+b_\ell}.\]
In particular, $\Lambda_f(x)>0$.
\end{theorem}


We now look more closely to the set of values attained by the $f$-Lyapunov exponent. Recall that $\phi=\frac{1+\sqrt{5}}{2}$ denotes the golden ration.

\begin{theorem}\label{thm:values_of_Lambda_f}
  For every $x\in \P^1(\R)$ we have $0\leq \Lambda_f(x)\leq \Lambda_f(\phi)$. Moreover, every value in $[0,\Lambda_f(\phi)]$ is attained by $\Lambda_f$. Equivalently, $\Lambda_f(\P^1(\R))=[0,\Lambda_f(\phi)]$.
\end{theorem}

We now turn our attention to a   special class of  quadratic irrationalities, namely the Markov irrationalities. They arise as  the worst irrational numbers with respect to rational approximation. The celebrated theorem  of Markov (see \cite[Section 2.2]{Aig2013}) establishes an explicit bijection between the $\GL$-equivalence classes of these worst irrational numbers, and sorted Markov triples\footnote{A Markov triple is a triple $(a, b, c)$ of positive integers satisfying Markov’s equation
$a^2 + b^2 + c^2 = 3abc.$}.  Markov irrationalities (and the  sorted Markov triples)  can be arranged in an infinite  tree, the Markov tree.  In terms of their purely periodic continued expansions,  the Markov tree  is constructed as follows: starting with the golden ratio $\phi=[\overline{1;1}]$ and the silver ratio $\psi=[\overline{2;2}]$, the Markov tree is generated by using the  conjunction of purely periodic continued fraction expansions
\[[\overline{a_1;\ldots;a_r}]\odot [\overline{b_1;\ldots ,b_s}]:=[\overline{a_1;\ldots,a_r,b_1,\ldots;b_s}],\]
in the form $w_1,w_2\mapsto w_2\odot w_1$, as illustrated in Figure \ref{fig:Markov_tree}. Here we use the abbreviations $[\overline{1;1}]=[\overline{1_2}]$, $[\overline{2;2}]=[\overline{2_2}]$, $[\overline{2;2,1,1}]=[\overline{2_2,1_2}]$, etc.

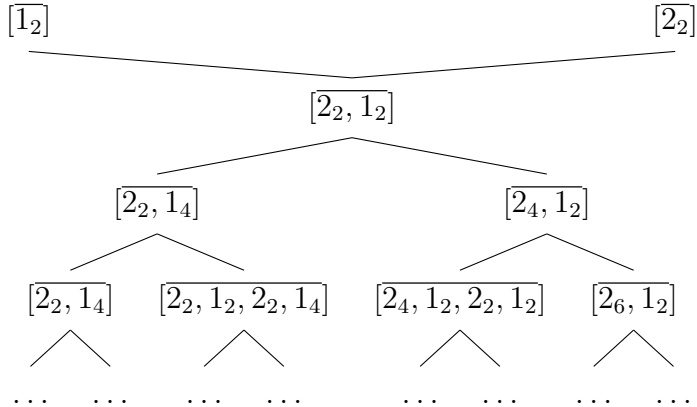
\begin{figure}[h!]
\centering
\begin{forest}
[,phantom [{$[\overline{1_2}]$},name=p1] [] [] [] [] [] [] 
[
[{$[\overline{2_2,1_2}]$},name=p2, no edge,tikz={\draw (p2.north)--(p1.south);}
[
{$[\overline{2_2,1_4}]$}
[
{$[\overline{2_2,1_4}]$}
[$\ldots$]
[$\ldots$]]
[
{$[\overline{2_2,1_2,2_2,1_4}]$}
[$\ldots$]
[$\ldots$]]
]
[
{$[\overline{2_4,1_2}]$}
[
{$[\overline{2_4,1_2,2_2,1_2}]$}
[$\ldots$]
[$\ldots$]] 
[
{$[\overline{2_6,1_2}]$}
[$\ldots$]
[$\ldots$]]]]]
[] [] [] [] [] []  [{$[\overline{2_2}]$},name=p3] tikz={\draw (p2.north)--(p3.south);}
]
\end{forest} \caption{The Markov tree.}\label{fig:Markov_tree}
\end{figure}

We denote by $\M$ the collection of all Markov irrationalities in this tree, which all have purely periodic continued fraction expansions. Our next theorem is the following. 

\begin{theorem}\label{thm:Lambda_f(M)}
 For every $w\in \M$, $\Lambda_f(w)\in [\Lambda_f(\psi),\Lambda_f(\phi)]$.   
\end{theorem}

In order to state our next result, we  introduce a natural parametrization of Markov irrationalities by Farey fractions in $[0,1/2]$: Let $\Fartwo$ denote the subtree of $\Far$ containing only the Farey fractions in $[0,1/2]$. Then, there is a unique monotonic increasing\footnote{The fact that bijection is increasing is also a consequence of Lemma \ref{lem: monotonicity Markov tree and opposites} in Section \ref{sec: convexity criteria}.} bijection $w:[0,1/2]\cap \mathbb{Q}\to \M$ with $w\left(\frac{0}{1}\right)= [\overline{1_2}]$, $w\left(\frac{1}{2}\right)= [\overline{2_2}]$, and satisfying
$$w\left(\frac{a}{b}\oplus \frac{c}{d}\right)=w\left(\frac{c}{d}\right)\odot w\left(\frac{a}{b}\right)$$
for all Farey fractions $\frac{a}{b}<\frac{c}{d}$ in $\Fartwo$ that are neighbors. We call this the Farey parametrization of $\M$. It is easy to see that this parametrization satisfies the following useful property\footnote{This property is the reason why we have chosen $\Fartwo$, among all possible Farey trees, to parametrize Markov irrationalities.}: if $p/q$ is a Farey fraction in $[0,1/2]$ and $w(p/q)=[\overline{a_1;\ldots,a_{\ell}}]$ with $\ell\geq 2$ even and minimal, then $a_1+\ldots+a_\ell=2q$. In particular, by Theorem \ref{thm:Lambda_f_on_quad_irrationals} we obtain
\[\Lambda_f\left(w\left(p/q\right)\right)=\frac{\Re(I_f(T^{a_1}\cdots V^{a_\ell}))}{2q}.\]
With this, we can now define the new function $\tilde \Lambda_f:[0,1/2]\cap \mathbb{Q}\to \R^+$ as
\[\tilde \Lambda_f\left(p/q\right):=\Lambda_f\left(w\left(p/q\right)\right).\]
The next theorem shows that $\tilde \Lambda_f$ has surprisingly good topological properties.  

\begin{theorem}\label{thm:Lambda_f_convexity}
    The map $\tilde \Lambda_f:[0,1/2]\cap \mathbb{Q}\to \R^+$ extends to a continuous, decreasing and convex function $\tilde \Lambda_f:[0,1/2]\to \R^+$.
\end{theorem}

\begin{remark}
    Note that $\Lambda_f$ and $\tilde \Lambda_f$ do not coincide in $[0,1/2]$. For instance, 
    $\Lambda_f(x)=0$ for every rational number $x\in [0,1/2]$ (by Theorem \ref{thm:first_properties_Lambda_f}) but $\tilde\Lambda_f(x)=\Lambda_f(w(x))\geq \Lambda_f(\psi)>0$. 
\end{remark}

For $f=j$ two plots of $\tilde\Lambda_j$ are shown in Figure \ref{fig: plots Lambdaj}. The values are between $\Lambda_j([\overline{2;2}])=625.68084367\ldots$ and $\Lambda_j([\overline{1;1}])=679.78370521\ldots$.

\begin{figure}[h!]
    \centering
    \subfloat{\includegraphics[width=0.45\linewidth]{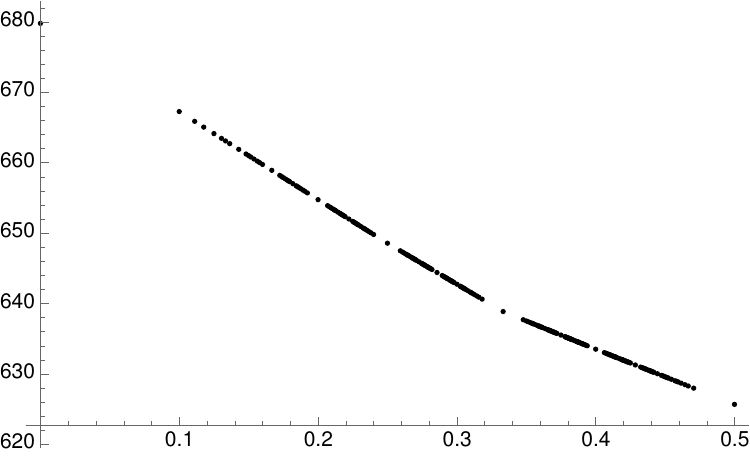}} \quad \quad 
    \subfloat{\includegraphics[width=0.45\textwidth]{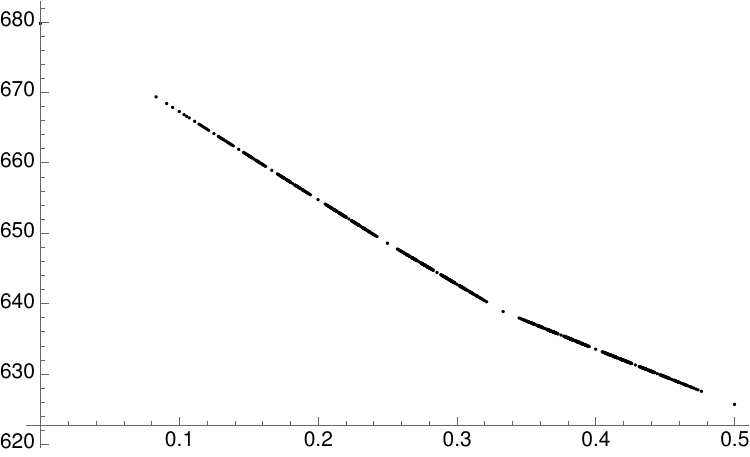}} 
\caption{Values of $\tilde \Lambda_j$ on the first 257 points in $\Fartwo$ on the left plot, and on the first 1025 points on the right plot.} \label{fig: plots Lambdaj}
\end{figure}

In the case $f=1$, our results specialize to the results of Spalding and Veselov on $\Lambda_1=2 \Lambda$. For the corresponding result  in the case of Theorem \ref{thm:values_of_Lambda_f}, they observe that the largest eigenvalue of a hyperbolic matrix can be compared to a matrix norm and   exploit the multiplicativity of such a norm, whereas in the case of Theorems \ref{thm:Lambda_f(M)} and \ref{thm:Lambda_f_convexity} they use that the Markov geodesics correspond to shortest simple geodesics on the \emph{modular torus} $\mathbb T=G\backslash \H$ with $G$ the commutator subgroup of $\SL$. They then make use of  the properties of the \emph{Federer--Gromov stable norm} on the first homology group of $\mathbb T$ as well as properties of a function due to Fock \cite{Foc98,FG07}. The connection of Fock's function to the stable norm was explained and used by Sorrentino and Veselov in \cite{SoV19} to study also the differentiability properties of Fock's function. 

\begin{remark}
    The relation between Markov numbers and the lengths of closed simple geodesics on $\mathbb{T}$ goes back to  Cohn \cite{Coh55} and Gorshkov \cite{Gor77}, and has been exploited by many people, notably by Series \cite{Ser85a,Ser85b}, McShane \cite{McS91}, McShane and Rivin \cite{MR95a,MR95b}, among others; see, e.g., \cite[p.~111]{Aig2013} and references therein.
\end{remark}

In our case of general modular function $f$,  since we do not have an obvious  natural  norm associated to $\Lambda_f$, we instead work directly with the definitions of cycle integrals and exploit the action of $\GL$  on the continued fraction expansion of quadratic irrationalities.

Finally, we present an immediate application of Theorem \ref{thm:Lambda_f_convexity} to Kaneko's  $\val$ function. Since
\[\Re(\val(w))=\frac{\Lambda_j(w)}{\Lambda_1(w)} \quad \text{for all quadratic irrationals }w,\]
if we use the Farey parametrization of the Markov irrationalities, as above, to define
\[\widetilde \val (p/q):=\Re(\val (w(p/q))) \quad \text{for }p/q\in [0,1/2]\cap \mathbb{Q},\]
then we obtain the following result.

\begin{coro}\label{coro:extended-val}
    $\widetilde{\val}$ extends to a continuous function $\widetilde \val:[0,1/2]\to \R^+$.
\end{coro}

\begin{remark}
\begin{enumerate}
    \item 
Kaneko also conjectured that  $\Re(\val)$ has an ``interlacing property'' on the Markov tree. This is equivalent to the assertion that   $\widetilde{\val} : [0,1/2] \to \R$ is a monotone increasing function. This would  imply in particular Kaneko's conjecture that for  any Markov quadratic $w$, $\Re(\val(w))\in [\val(\phi), \val(\psi)]$. Unfortunately the interlacing conjecture of Kaneko does not follow   from  the monotonicity of $\Lambda_j$ and $\Lambda_1$, since both functions are monotone decreasing. However we do get that
       \[650.1095\ldots=\frac{\Lambda_j(\psi)}{\Lambda_1(\phi)}\leq \Re(\val(w))\leq \frac{\Lambda_j(\phi)}{\Lambda_1(\psi)}=771.2776\ldots,\text{ for all }w\in \M.\]
       Compare with $\val(\phi)=706.3248\ldots$ and $\val(\psi)=709.8928\ldots$. 
       We hope to come back to this problem in the near future.
       \item Corollary \ref{coro:extended-val} gives a new proof of \cite[Theorem 1.1]{BI19} for $\Re(\val(w))$. 
    \item  Murakami \cite{Mur21} also studied limit values of cycle integrals normalized by hyperbolic length, in particular of Kaneko's $\val$ function, along certain sequences of quadratics irrationalities. Murakami's work however goes in a different direction.
       \end{enumerate}
\end{remark} 
\subsection{Outline of the paper}
The rest of the paper is organized as follows. In the next section  we introduce the notation and give the  properties of the cycle integrals that we will need for the proofs of our main theorems. In Section \ref{sec:first_properties_of_Lambda}  we prove Theorems \ref{thm:first_properties_Lambda_f} and \ref{thm:Lambda_f_on_quad_irrationals}. In  Section \ref{sec:Values attained by Lambda_f}  we prove Theorem\ref{thm:values_of_Lambda_f}. In the last two sections we restrict to the Markov irrationalities and prove Theorems \ref{thm:Lambda_f(M)} and Theorem\ref{thm:Lambda_f_convexity}.

\section*{Acknowledgments}
The authors thank Prof.~Veselov for comments and references on the connection between Fock's function and Federer--Gromov stable norm.

\section{Cycle integrals}\label{sec:cycle_integrals}

As a convention, we extend the definition of $I_f(A)$ given in \eqref{eq:def_cycle_integral_intro} for  a non-hyperbolic $A\in \SL$    by putting $I_f(A):=0$ for such  $A$.

\subsection{Elementary properties}

We start with the following remark which will be used several times. 

\begin{remark}\label{rmk:mod_functions_with_real_fourier_coefficients}
   Given a weakly holomophic modular function $f:\H\to \C$ for $\SL$, the following three conditions are easily seen to be equivalent:
   \begin{enumerate}
       \item[$(i)$] $f$ has real Fourier coefficients.
       \item[$(ii)$] $f(-\overline{\tau})=\overline{f(\tau)}$ for all $\tau\in \H$.
       \item[$(iii)$] $f(e^{it})\in \R$ for all $t\in [\pi/3,2\pi/3]$.
   \end{enumerate}
\end{remark}

 The next lemma collects well known properties of the cycle integrals of modular functions. We skip its proof since it is similar to the one given in \cite{Kan09} for the $j$ function.
 
\begin{lemma}\label{lem:elementary_prop_cycle_integrals}
    Let $f$ be a weakly holomophic modular function for $\SL$ and $A\in \SL$. Then, the following properties hold:
    \begin{enumerate}
        \item[$(i)$] $I_f(-A)=I_f(A^{-1})=I_f(A)$.
        \item[$(ii)$]  $I_f(A^n)=|n|I_f(A)$ for all $n\in \Z$.
        \item[$(iii)$]  $I_f(MAM^{-1})=I_f(A)$ for all $M\in \SL$.
        \item[$(iv)$]  If $f$ has real Fourier coefficients, then $I_f(MAM^{-1})= \overline{I_f(A)}$ for all $M\in \GL$ with $\det(M)=-1$.
    \end{enumerate}
\end{lemma}

\begin{remark}\label{rmk:starting_with_T_ending_V}
It is well known\footnote{This follows, e.g., from \cite[(II) in p.~4]{LZ97}} that every hyperbolic matrix $A\in \SL$ is, up to a sign, conjugated in $\SL$ to a matrix in $\SLN$, hence conjugated to an element of the form $\pm T^{a_1}V^{a_2}\cdots V^{a_\ell}$ with $a_1,\ldots,a_\ell\geq 0$. Hence, if $f$ is a weakly holomorphic modular function for $\SL$ with real Fourier coefficients, then from Lemma \ref{lem:elementary_prop_cycle_integrals} and the identity
\begin{equation}\label{eq:STS=V}
    \begin{pmatrix}
    0 & 1 \\ 1 & 0
\end{pmatrix}T\begin{pmatrix}
    0 & 1 \\ 1 & 0
\end{pmatrix}^{-1}=V
\end{equation}
we deduce that, 
for the computation of $\Re(I_f(A))$ 
we can restrict ourselves to hyperbolic matrices of the form $A=T^{a_1}V^{a_2}\cdots V^{a_\ell}$ with $a_1,\ldots,a_\ell\geq 1$ and $\ell\geq 2$ even. 
\end{remark}

\subsection{Cycle of a purely periodic quadratic irrational}

Let $\ell \geq 2$ even, $a_1,\ldots,a_\ell\geq 1$ integers and $A=T^{a_1}V^{a_2}\cdots T^{a_{\ell-1}}V^{a_\ell}$. The repelling and attracting fix points of $A^{-1}$ are 
\[w=[\overline{a_1;\ldots,a_\ell}] \, \text{ and } \, \tilde{w}=-1/[\overline{a_\ell;\ldots,a_1}],\]
respectively. In particular, $\tilde{w}<0<w$, so we refer to $w$ and $\tilde{w}$ as the positive and negative fixed points of $A$, respectively.

The action of $T^{-1}$ on $[\overline{a_1;\ldots,a_\ell}]$ gives $[a_1-1;\overline{a_2,\ldots,a_\ell,a_1}]$ hence 
\(T^{-a_1}[\overline{a_1;\ldots,a_\ell}]=[0;\overline{a_2,\ldots,a_\ell,a_1}].\)
Similarly, the action of $V^{-1}$ on $[0;\overline{a_2,\ldots,a_\ell,a_1}]$ gives $[0;a_2-1,\overline{a_3,\ldots,a_\ell,a_1,a_2}]$ hence 
\(V^{-a_2}[0;\overline{a_2,\ldots,a_\ell,a_1}]=[0;0,\overline{a_3,\ldots,a_\ell,a_1,a_2}]=[\overline{a_3;\ldots,a_\ell,a_1,a_2}].\)
Acting by $T^{-1}$ and $V^{-1}$ iteratively in this way gives a sequence of positive real quadratic irrationals $(w^{(k)})_{k\geq 1}$ with
	\begin{equation}\label{eq:w_cyclic_sequence}
		w^{(1)}=w,\qquad	w^{(k+1)}=\left\{\begin{array}{ll}
			T^{-1} (w^{(k)}) &\text{if }w^{(k)}> 1 ,\\
			V^{-1} (w^{(k)}) &\text{if }0<w^{(k)}< 1.
		\end{array}\right.
	\end{equation}
This sequence is cyclic of length 
\begin{equation}\label{sn}
s(A)=s(w):= a_1+\ldots+a_\ell.
\end{equation}
We denote by $\tilde{w}^{(k)}$ the Galois conjugate of $w^{(k)}$ and define
\begin{eqnarray*}
K_{w,r,s }(t)&:=& \sum_{k=r}^{s}\dfrac{1}{e^{it}-w^{(k)}}  - \dfrac{1}{e^{it}-\tilde w^{(k)}} \quad \text{for }s\geq r\geq 1, \text{ and}\\
K_w(t)&:=&K_{w,1,s(w)}(t).
\end{eqnarray*}

\begin{remark}\label{remark:shape_of_wk_and_conjugates}
Each term $w^{(k)}$ is of the form
\begin{equation}\label{eq:w^k}
    \begin{cases}
    [j;\overline{a_{i+1},a_{i+2},\ldots,a_\ell,a_1,\ldots,a_i}] & \text{with }1\leq j\leq a_i \text{ and odd }i\in\left\{1,\ldots,\ell\right\}, \text{ or}\\
[0;j,\overline{a_{i+1},a_{i+2},\ldots,a_\ell,a_1,\ldots,a_i}] & \text{with }1\leq j\leq a_i \text{ and even }i\in\left\{1,\ldots,\ell\right\}.
\end{cases}\end{equation}
Similarly, each term $\tilde w^{(k)}$ is of the form
\begin{equation}\label{eq:tilde_w^k}
    \begin{cases}
    -[a_{i}-j;\overline{a_{i-1},\ldots,a_1,a_\ell,\ldots,a_{i}}] & \text{with }1\leq j\leq a_i \text{ and odd }i\in\left\{1,\ldots,\ell\right\}, \text{ or}\\
-[0;a_i-j,\overline{a_{i-1},\ldots,a_1,a_\ell,\ldots,a_i}] & \text{with }1\leq j\leq a_i \text{ and even }i\in\left\{1,\ldots,\ell\right\}.
\end{cases}
\end{equation} 
\end{remark}

We have the following result which is a consequence of \cite[Lemma 4.1]{BI20}.

\begin{lemma}\label{lem:indentity_cycle_int_BI}
Let $A=T^{a_1}V^{a_2}\cdots T^{a_{\ell-1}}V^{a_\ell}$ with $n\geq 2$ even and $a_1,\ldots,a_\ell\geq 1$, and $(w^{(k)})_{k\geq 0}$ the sequence given by \eqref{eq:w_cyclic_sequence}. Then:
\begin{equation}
I_f(A)=\int_{\pi/3}^{2\pi/3} f(e^{it})  K_{w}(t) ie^{it}dt.  
\end{equation}
\end{lemma}

Using Lemma \ref{lem:indentity_cycle_int_BI} we now prove the following simple estimate.

\begin{lemma}\label{lem:trivial_bound_I_f(A)}
Let $A=T^{a_1}V^{a_2}\cdots T^{a_{\ell-1}}V^{a_\ell}$ with $\ell\geq 2$ even, $a_1,a_{\ell}\geq 0$ and $a_2,\ldots,a_{\ell-1}\geq 1$. Then, we have
\[I_f(A)=O(a_1+\ldots+a_\ell)\]
with an absolute implicit constant. In particular, for all $x\in (0,\infty)$ the limit superior \eqref{eq:def_lambda_f_intro} is finite.
\end{lemma}
\begin{proof}
We can assume that $A$ is hyperbolic. Moreover, by Remark \ref{rmk:starting_with_T_ending_V} we can assume $a_1\geq 1$ and $a_{\ell}\geq 1$. Then the result follows from Lemma \ref{lem:indentity_cycle_int_BI} by applying the trivial bound
\begin{equation}\label{trivial bound}
\frac{1}{|e^{\pm it}-x|}\leq \frac{2}{\sqrt{3}} \quad \text{for all }t\in \left[\pi/3,2\pi/3\right] \text{ and } x\in \R
\end{equation}
to each term in the sum $K_w(t)$. 
This proves the lemma.
\end{proof}


\subsection{Almost invariance}

In this section we prove   that along a path in the Farey tree, the cycle integrals of the associated hyperbolic matrices are almost invariant under multiplication by matrices in $\SLN$ up to a bounded error term. This will be used later to prove that $\Lambda_f$ is invariant under the action of $\SLN$ on $(0,\infty)$. More precisely we have 
the following theorem.
\begin{theorem}\label{thm:almost_invariance_cycle_integrals}
		Let $(A_n)_{n\geq 0}$ be the sequence of matrices in $\SLN$ associated to a path in the Farey tree $\mathcal{F}_{\frac{0}{1},\frac{1}{0}}$ converging to a real number $x\in (0,\infty)$. Then, for every $M\in \SLN$ and $n\geq 0$ we have
\begin{equation*}
        I_f(MA_n)=I_f(A_n)+\left(\sup_{t\in [\pi/3,2\pi/3]}|f(e^{it})|\right)\cdot O_{x,M}\left(1\right),
  \end{equation*}
  where the term $O_{x,M}(1)$ depends only on $x$ and $M$.
\end{theorem}

We start with the following lemma whose proof can be found  in \cite[Lemma 1.24]{Aig2013}.

\begin{lemma}\label{lemacoincide}
	If the continued fraction expansions of two irrational real numbers $x$ and $y$ coincide in the first $k$ partial quotients, namely $x=[a_1;a_2,\ldots,a_k,\ldots ]$, $y=[b_1;b_2,\ldots,b_k,\ldots ]$ with $a_i=b_i$ for $i=1,\ldots,k$, then
	$$
	|x-y|\leq \dfrac{1}{2^{k-2}}.
	$$
\end{lemma}

We will also need the following proposition.

\begin{prop}\label{prop:general_comparison_K_u_m_1vsK_v_m_2}
 Let $u,v$ be two real quadratic irrationals with purely periodic continued fraction expansions $u=[\overline{a_1;\ldots ,a_r}]$, $v=[\overline{b_1;\ldots, b_s}]$ 
     where $a_i,b_j\geq 1$ for all $i,j$. Given $1\leq \ell \leq \min\{r,s\}$ put $m_1:=a_1+\ldots+a_\ell$ and $m_2:=b_1+\ldots+b_\ell$. Then, for every $t\in \left[\pi/3,2\pi/3\right]$ we have
\begin{equation*}
   |K_{u,1,m_1}(t)- K_{v,1,m_2}(t)|=O\left(1+\sum_{i=1}^\ell|a_i-b_i|\right)
\end{equation*}
with an absolute implicit constant.
\end{prop}
\begin{proof}
Let $0\leq p\leq \ell$ be the largest integer such that $a_i=b_i$ for all $i\in \{1,\ldots,p\}$. We can then write \(u=[\overline{a_1;\ldots ,a_p,d_1,\ldots,d_{r-p}}]\), \(v=[\overline{a_1;\ldots, a_p,e_1,\ldots,e_{s-p}}]\) where $a_i,d_j,e_k\geq 1$ for all $i,j,k$. Put $m:=a_1+\ldots +a_p$. It is enough to prove that for all $t\in \left[\pi/3,2\pi/3\right]$ we have
\begin{equation}\label{eq:difference_K_u,m_and_K_v,m}
   |K_{u,1,m}(t)- K_{v,1,m}(t)|=O(1), 
\end{equation}
and 
\begin{equation}\label{eq:difference_K_u',b1_and_K_v',c1}
|K_{u,m+1,m+d_1}(t)- K_{v,m+1,m+e_1}(t)|=O(|d_1-e_1|) \quad \text{if }p<\ell.  
\end{equation}
    We write $K_{u,1,m}(t)- K_{v,1,m}(t)=K_1(t)+K_2(t)$, where
\begin{align*}
			&K_1(t) := \sum_{k=1}^{m} \dfrac{1}{e^{it}-u^{(k)}} - \dfrac{1}{e^{it}-v^{(k)}}
			, \\
			&K_2(t) := \sum_{k=1}^{m}   \dfrac{1}{e^{it}-\tilde v^{(k)}} - \dfrac{1}{e^{it}-\tilde u^{(k)}}.
		\end{align*}
   Since 
   \begin{equation*}
       |e^{\pm it}-x|\geq \frac{x}{2} \quad \text{for all }t\in \left[\pi/3,2\pi/3\right] \text{ and } x\in [1,\infty),
   \end{equation*}
   and
   		\begin{equation*}
			\label{compuv}
			\left|\dfrac{1}{e^{it}- u^{(k)}} - \dfrac{1}{e^{it}- v^{(k)}}\right| 
			= \dfrac{|u^{(k)}-v^{(k)}|}{|e^{it}-u^{(k)}||e^{it}-v^{(k)}|}
			= \dfrac{\left|\frac{1}{u^{(k)}}-\frac{1}{v^{(k)}}\right|}{\left|e^{-it}-\frac{1}{u^{(k)}}\right| \left|e^{-it}-\frac{1}{v^{(k)}}\right|},
		\end{equation*}
   we get
   		\begin{equation}\label{eq:comparison_two_cycles_of_irrationals}
			\left|\dfrac{1}{e^{it}- u^{(k)}} - \dfrac{1}{e^{it}- v^{(k)}}\right| 
			\leq 4\times  \left\{
   \begin{array}{ll}
      \dfrac{|u^{(k)}-v^{(k)}|}{|u^{(k)}| |v^{(k)}|}   & \text{ if }u^{(k)},v^{(k)}\in [1,\infty), \\
      \dfrac{\left|\frac{1}{u^{(k)}}-\frac{1}{v^{(k)}}\right|}{\left|\frac{1}{u^{(k)}}\right| \left|\frac{1}{v^{(k)}}\right|}& \text{ if } u^{(k)},v^{(k)}\in (0,1].
   \end{array}\right.
		\end{equation}
For $1\leq k\leq a_1$, the terms $u^{(k)}$ and $v^{(k)}$ are in $[a_1+1-k,\infty)$ and share the first $p$ partial quotients, hence by \eqref{eq:comparison_two_cycles_of_irrationals} and Lemma~\ref{lemacoincide} we have
  \[\sum_{k=1}^{a_1-1} \left|\dfrac{1}{e^{it}- u^{(k)}} - \dfrac{1}{e^{it}- v^{(k)}}\right| \leq \frac{4}{2^{p-2}}\sum_{k=1}^{a_1}\frac{1}{(a_1+1-k)^2}\leq \frac{\zeta(2)}{2^{p-4}}.\]
For $a_1+1\leq k\leq a_1+a_2$, the terms $u^{(k)}$ and $v^{(k)}$ are in $\left(0,\frac{1}{a_2+1-(k-a_1)}\right]$ and $1/u^{(k)},$  $1/v^{(k)}$ share the first $p-1$ partial quotients, thus in this case we get
\[\sum_{k=a_1+1}^{a_1+a_2} \left|\dfrac{1}{e^{it}- u^{(k)}} - \dfrac{1}{e^{it}- v^{(k)}}\right| \leq \frac{4}{2^{p-3}}\sum_{k=a_1+1}^{a_1+a_2}\frac{1}{(a_2+1-(k-a_1))^2}\leq \frac{\zeta(2)}{2^{p-5}}.\]
By repeating this argument for the next terms $u^{(k)}$ and $v^{(k)}$ for $a_1+a_2+1\leq k\leq a_1+a_2+a_3$, etc., we conclude
	\begin{equation*}
		|K_1(t)|\leq \zeta(2)\sum_{i=1}^p \frac{1}{2^{p-(3+i)}}=O(1).
	\end{equation*}  
Similarly, for $2\leq k\leq a_1+1$ the terms $-\tilde u^{(k)},-\tilde v^{(k)}$ are in $[k-1,\infty)$ and share the first partial quotient; for $a_1+2\leq k\leq a_1+a_2+1$ they are in $\left(0,\frac{1}{k-a_1-1}\right]$ and  $-1/\tilde u^{(k)},-1/\tilde v^{(k)}$ share the first two partial quotients, etc., hence we get
	\begin{equation*}
	|K_2(t)|\leq\zeta(2)\sum_{i=1}^p \frac{1}{2^{i-4}}=O(1).
	\end{equation*}
This implies \eqref{eq:difference_K_u,m_and_K_v,m}. In order to prove \eqref{eq:difference_K_u',b1_and_K_v',c1} we  first assume $p$ is even. Let \(w:=u^{(m+1)}=[\overline{d_1;\ldots ,d_{r-p},a_1,\ldots,a_p}\) and \(z:=v^{(m+1)}=[\overline{e_1;\ldots ,e_{s-p},a_1,\ldots,a_p}]\). We write
\[K_{u,m+1,m+d_1}(t)-K_{v,m+1,m+e_1}(t)=K_{w,1,d_1}(t)-K_{z,1,e_1}(t)=K_3(t)+K_4(t)+K_5(t)+K_6(t),\]
where
\begin{eqnarray*}
K_3(t) &:=& \sum_{k=0}^{\min\{d_1,e_1\}-1} \dfrac{1}{e^{it}-w^{(d_1-k)}} - \dfrac{1}{e^{it}-z^{(e_1-k)}}
			, \\
K_4(t) &:=& \sum_{k=\min\{d_1,e_1\}}^{d_1-1} \dfrac{1}{e^{it}-w^{(d_1-k)}} -\sum_{k=\min\{d_1,e_1\}}^{e_1-1} \dfrac{1}{e^{it}-z^{(e_1-k)}}
			, \\
   K_5(t) &:=& \sum_{k=2}^{\min\{d_1,e_1\}} \dfrac{1}{e^{it}-\tilde z^{(k)}} - \dfrac{1}{e^{it}-\tilde w^{(k)}}
			, \\
   K_6(t) &:=& \dfrac{1}{e^{it}-\tilde z^{(1)}} - \dfrac{1}{e^{it}-\tilde w^{(1)}} + \sum_{k=\min\{d_1,e_1\}+1}^{e_1} \dfrac{1}{e^{it}-\tilde z^{(k)}} - \sum_{k=\min\{d_1,e_1\}+1}^{d_1}\dfrac{1}{e^{it}-\tilde z^{(k)}}.
		\end{eqnarray*}
For $0\leq k\leq \min\{d_1,e_1\}-1$ the terms $w^{(d_1-k)}$ and $z^{(e_1-k)}$  are in $[k+1,\infty)$ and share the first partial quotient. This implies
\(|K_3(t)|=O(1),\)
again by \eqref{eq:comparison_two_cycles_of_irrationals} and Lemma~\ref{lemacoincide}. Similarly, for $2\leq k\leq \min\{d_1,e_1\}$ the terms $-\tilde w^{(k)}$ and $-\tilde z^{(k)} $  are in $[k-1,\infty)$ and share the first partial quotient, hence we also get
\(|K_5(t)|=O(1).\)
Finally, using the trivial estimate \eqref{trivial bound} we bound $K_4(t)$ and $K_6(t)$ by the number of terms in these sums, obtaining
\(|K_4(t)|+|K_6(t)|=O(|d_1-e_1|).\)
This shows \eqref{eq:difference_K_u',b1_and_K_v',c1} when $p$ is even. The case when $p$ is odd is similar, so we omit the details for brevity. This completes the proof of the proposition.
\end{proof}


We now give the proof of Theorem \ref{thm:almost_invariance_cycle_integrals}.

\begin{proof}[Proof of Theorem \ref{thm:almost_invariance_cycle_integrals}]
    Since $\SLN$ is generated by $T$ and $V$, it is enough to prove the result for $M\in \{T,V\}$. Moreover, we can assume that $n$ is large enough so that $A_n$ is hyperbolic. Write $A_n=T^{a_1}V^{a_2}\cdots T^{a_{\ell-1}}V^{a_{\ell}}$ with $\ell\geq 2$ even, $a_1,a_{\ell}\geq 0$ and $a_2,\ldots,a_{\ell-1}\geq 1$, and let $M=T$.

    If $a_1,a_{\ell}\geq 1$, then the positive fixed point of $A_n$ is $u:=[\overline{a_1;\ldots,a_{\ell}}]$, while the positive fixed point of $TA_n$ is $v:=[\overline{a_1+1;\ldots,a_{\ell}}]$. Then, a direct application of Proposition \ref{prop:general_comparison_K_u_m_1vsK_v_m_2} gives
    \(|K_{u}(t)-K_v(t)|=O(1).\)
    Then, by Lemma \ref{lem:indentity_cycle_int_BI} we obtain the desired result.

     If $a_1\geq 1$ and $a_{\ell}=0$, then $T^{a_{\ell-1}}A_nT^{-a_{\ell-1}}$ and $T^{a_{\ell-1}}(TA_n)T^{-a_{\ell-1}}$ have positive fixed points $[\overline{a_1+a_{\ell-1};\ldots,a_{\ell-2}}]$ and $[\overline{a_1+a_{\ell-1}+1;\ldots,a_{\ell-2}}]$. Then, using the conjugacy invariance of the cycle integrals (Lemma \ref{lem:elementary_prop_cycle_integrals}$(iii)$)  together with Lemma \ref{lem:indentity_cycle_int_BI} and  Proposition \ref{prop:general_comparison_K_u_m_1vsK_v_m_2}, we get
     \[I_f(A_n)-I_f(TA_n)=I_f(T^{a_{\ell-1}}A_nT^{-a_{\ell-1}})-I_f(T^{a_{\ell-1}}(TA_n)T^{-a_{\ell-1}})=O(1)\cdot \sup_{t\in [\pi/3,2\pi/3]}|f(e^{it})|.\]

     If $a_1=0$ and $a_{\ell}\geq 1$, then $V^{-a_2}A_nV^{a_2}$ and $(TV^{a_2})^{-1}(TA_n)(TV^{a_2})$ have positive fixed points $w:=[\overline{a_3;\ldots,a_{\ell-1},a_{\ell}+a_2}]$ and $z:=[\overline{a_3;\ldots,a_{\ell},1,a_2}]$. In this case, Proposition \ref{prop:general_comparison_K_u_m_1vsK_v_m_2} together with the trivial bound \eqref{trivial bound}, imply
     \(|K_{w}(t)-K_z(t)|=O(1+a_2),\) 
     which is an estimate depending only on $x$. Then, the desired result follows as in the previous case by using Lemmas \ref{lem:elementary_prop_cycle_integrals}$(iii)$ and \ref{lem:indentity_cycle_int_BI}.

     Finally, if $a_1=a_{\ell}=0$, then $V^{-a_2}A_nV^{a_2}$ and $(TV^{a_2})^{-1}(TA_n)(TV^{a_2})$ have positive fixed points $[\overline{a_3;\ldots,a_{\ell-1},a_2}]$ and $[\overline{a_3;\ldots,a_{\ell-2},a_{\ell-1}+1,a_2}]$, and the result follows as in the first two cases.

     We conclude that the desired results holds for $M=T$. The case $M=V$ is completely analogous, hence we omit the details. This finishes our proof.
\end{proof}

\subsection{Almost additivity}

In this section we prove the next theorem which is crucial for the proof of Theorems \ref{thm:Lambda_f_on_quad_irrationals} and \ref{thm:values_of_Lambda_f}.  It can be interpreted as  an almost additivity of cycle integrals along Farey paths associated to   badly approximable numbers. Recall that badly approximable numbers are characterized by having bounded partial quotients (see, e.g., \cite[Proposition 1.32]{Aig2013}). 

\begin{theorem}\label{thm:almost_additivity_for_badly_approximable}
		Let $(A_n)_{n\geq 0}$ be the sequence of matrices in $\SLN$ associated to the irrational real number $x\in (0,\infty)$ with infinite continued fraction expansion $x=[a_1;a_2,\ldots ]$ as in \eqref{eq:hyp_matrix_seq_intro}. Assume that the sequence $(a_n)_{n\geq 1}$ is bounded by a constant $C>0$. For $m,n\geq 0$ write $A_{m+n}=A_nB_{m,n}$. Then, we have
\begin{equation*}
        I_f(A_{m+n})=I_f(A_n)+I_f(B_{m,n})+\left(\sup_{t\in [\pi/3,2\pi/3]}|f(e^{it})|\right)\cdot O\left(C\right)
  \end{equation*}
  with an absolute implicit constant.
\end{theorem}

\begin{proof}
There are different cases to be considered. For the sake of brevity, we only show how to proceed in two of them.

 As a first case, we assume that both $A_n$ and $B_{m,n}$ are hyperbolic and end with a power of $V$. More precisely, $A_n=T^{a_1}V^{a_2}\cdots T^{a_{r-1}}V^j$ with $1\leq j\leq a_r$ and $B_{m,n}=V^{a_r-j}T^{a_{r+1}}\cdots T^{a_{s-1}}V^k$ with $1\leq k\leq a_s$, where  $r,s$ are even and $2\leq r<s$. Put \( M=a_1+\ldots+a_{r},  N=a_{r+1}+\ldots+a_{s-1}+k\)
and let \(u=[\overline{a_1;\ldots,a_{r-1},a_r,\ldots,a_{s-1},k}], v=\overline{a_1;\ldots,a_{r-1},j}], w=[\overline{a_{r+1};\ldots,a_{s-1},k+a_r-j}]\)
be the positive fixed points of $A_{m+n}$, $A_n$ and $V^{j-a_r}B_{m,n}V^{a_r-j}$, respectively. By Lemma \ref{lem:elementary_prop_cycle_integrals}$(iii)$ we have $I_f(V^{j-a_r}B_{m,n}V^{a_r-j})=I_f(B_{m,n})$. Then, by Lemma \ref{lem:indentity_cycle_int_BI} we get
\begin{eqnarray*}
I_f(A_{m+n})-I_f(A_n)-I_f(B_{m,n})
&=&\int_{\pi /3}^{2\pi /3} f(e^{it}) (K_{u}(t) - K_{v}(t)-K_{w}(t))ie^{it}dt.
\end{eqnarray*}
Now, we write
\(K_{u}(t)=K_{u,1,M}(t)+K_{u',1,N}(t),\)
where $u'=[\overline{a_{r+1},\ldots,a_{s-1},k,a_1,\ldots,a_{r}}]$. Then, by Proposition \ref{prop:general_comparison_K_u_m_1vsK_v_m_2} we have
\[|K_{u,1,M}(t)-K_{v}(t)| = O(1+(a_r-j)),\quad 
|K_{u',1,N}(t)-K_{w}(t)| = O(1+(a_r-j)).\]
Since $a_r-j\leq a_r\leq C$, we conclude
\[I_f(A_{m+n})-I_f(A_n)-I_f(B_{m,n})=\left(\sup_{t\in [\pi/3,2\pi/3]}|f(e^{it})|\right)\cdot O(C),\]
as desired.

As a second case, we assume $A$ ends with a $T$, $B_{m,n}$ starts with a $T$ and ends with a $V$. This is, $A_n=T^{a_1}V^{a_2}\cdots V^{a_{r-1}}T^j$ with $1\leq j<a_r$ and $B_{m,n}=T^{a_r-j}V^{a_{r+1}}\cdots T^{a_{s-1}}V^k$ with $1\leq k\leq a_s$, where  $r$ is odd, $s$ is even and $3\leq r<s$. In this case we put 
\( M=a_1+\ldots+a_{r-1}, N=a_{r}+\ldots+a_{s-1}+k\)
and
\(u=[\overline{a_1;\ldots,a_{r-1},a_r,\ldots,a_{s-1},k}],
v=[\overline{a_1+j;a_2,\ldots,a_{r-1}}],
w=[\overline{a_{r}-j;a_{r+1},\ldots,a_{s-1},k}]\) 
the positive fixed points of $A_{m+n}$, $T^jA_nT^{-j}$ and $B_{m,n}$, respectively. Writing
\(K_{u}(t)=K_{u,1,M}(t)+K_{u',1,N}(t),\)
with $u'=[\overline{a_{r},\ldots,a_{s-1},k,a_1,\ldots,a_{r-1}}]$ we get, by Proposition \ref{prop:general_comparison_K_u_m_1vsK_v_m_2}, the estimates
\[|K_{u,1,M}(t)-K_{v}(t)| = O(1+j),\quad
|K_{u',1,N}(t)-K_{w}(t)| = O(1+j).\]
Since $j<a_r\leq C$, the desired result follows  using Lemmas \ref{lem:elementary_prop_cycle_integrals}$(iii)$ and \ref{lem:indentity_cycle_int_BI} as in the first case. The rest of the cases are very similar and are left to the careful reader. This concludes the proof of the theorem.
\end{proof}

\begin{remark}
A special case of Theorem \ref{thm:almost_additivity_for_badly_approximable}, namely for $x$ quadratic irrational, was proved by Matsusaka in \cite{Matsusaka}.
\end{remark}

\begin{coro}\label{cor:I(A_m+n)-I(A_n)_goes_to_zero}
    Let $(A_n)_{n\geq 0}$ be as in Theorem \ref{thm:almost_additivity_for_badly_approximable} associated to a badly approximable $x\in (0,\infty)$. Then, for every $m\geq 0$ we have
    \[\lim_{n\to \infty}\left(\frac{I_f(A_{n+m})}{n+m}-\frac{I_f(A_{n})}{n}\right)=0.\]
\end{coro}
\begin{proof}
    Writing $A_{n+m}=A_nB_{m,n}$ we have
\begin{align*}
    \frac{I_f(A_{n+m})}{n+m}-\frac{I_f(A_{n})}{n}&=\frac{n(I_f(A_n)+I_f(B_{m,n})+O_{f}(C))-(n+m)I_f(A_{n})}{(n+m)n}\\
    &= \frac{I_f(B_{m,n})+O_{f}(C)}{n+m}-\frac{mI_f(A_{n})}{(n+m)n}
\end{align*}    
where the notation is as in Theorem \ref{thm:almost_additivity_for_badly_approximable} and the implicit constant in $O_f(C)$ depends only on $f$. But Lemma \ref{lem:trivial_bound_I_f(A)} gives $I_f(B_{m,n})=O(m)$ and $I_f(A_{n})=O(n)$ with absolute implicit constants. Hence, letting $n\to \infty$ we get the desired result.
\end{proof}

\subsection{A formula for $\Re(I_f(A))$}

When $f(e^{it})$ is real for $t\in [\pi/3,2\pi/3]$, a computation of  $\Re(I_f(A))$ using Lemma \ref{lem:indentity_cycle_int_BI} naturally leads to  the function $F:\R\times [\pi/3,2\pi/3]\to \R$ defined by
\begin{equation}\label{eq:def_of_F(x,t)}
    F(x,t):=\frac{1}{\sin(t)}\Re\left(\frac{i e^{it}}{e^{it}-x}\right)=\frac{x}{1+x^2-2x\cos(t)},
\end{equation}
since
\begin{equation}\label{eq:identity_function_F}
\Re\left(\left(\frac{1}{e^{it}-x}-\dfrac{1}{e^{it}-y}\right)ie^{it}\right)=\sin(t)\big(F(x,t)-F(y,t)\big).
\end{equation}

The basic properties of the function $F(x,t)$ are summarized in the following lemma, whose proof we omit since they follow from direct computations.

\begin{lemma}\label{lem:prop_of_F}
For all $x\in \R$ and $t\in [\pi/3,2\pi/3]$ we have the following properties:
\begin{enumerate}
\item[$(i)$] $F(-x,t)=-F(x,\pi-t)$.
\item[$(ii)$] If $x\neq 0$ then $F(1/x,t)=F(x,t)$.
\item[$(iii)$] $\partial_x F(x,t)=\dfrac{1-x^2}{(1 + x^2 - 
  2 x \cos(t))^2}$.
\end{enumerate}
In particular, for every fixed $t\in [\pi/3,2\pi/3]$ the function $x\mapsto F(x,t)$ is increasing in $[0,1]$ and decreasing in $[1,\infty)$.   
\end{lemma}

The plots of $x\mapsto F(x,t)$ for $t\in \left\{\pi/3,\pi/2,2\pi/3\right\}$ are shown in Figure \ref{fig:plot of F}.

\begin{figure}[h!]
    \centering
    \includegraphics[width=0.5\linewidth]{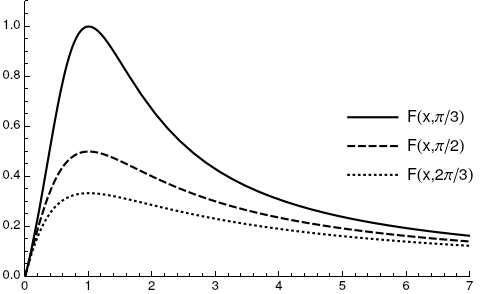}
    \caption{Plot of $x\mapsto F(x,t)$ for $t\in \left\{\pi/3,\pi/2,2\pi/3\right\}$.} \label{fig:plot of F}
\end{figure}

In order to state our next result, we introduce the following notation: For a quadratic irrationality $w$ with purely periodic continued fraction expansion $w=[\overline{a_1;a_2,\ldots,a_\ell}]$  and $t\in [\pi/3,2\pi/3]$ we define
\begin{equation}\label{eq:def_of_S(w,t)}
S(w,t):=\sum_{i=1}^{\ell}\sum_{j=1}^{a_i}F([j;\overline{a_{i+1},\ldots,a_\ell,a_1,\ldots,a_i}],t).
\end{equation}
Using the well known fact that
\(w^{\op}:=[\overline{a_\ell;a_{\ell-1},\ldots,a_1}]=-1/\tilde{w}\) 
(see, e.g., \cite[Lemma 1.28]{Aig2013}) we get
\begin{equation}\label{def_of_S(-1/tilde w,t)}
S(-1/\tilde{w},t)=S(w^{\op},t)=\sum_{i=1}^\ell\sum_{j=1}^{a_i}F([j;\overline{a_{i},\ldots,a_1,a_\ell,\ldots,a_{i+1}}],t).
\end{equation}

\begin{prop}\label{prop:formula_for_Re(I_f(A))}
Let $f$ be a weakly holomophic modular function for $\SL$ with real Fourier coefficients. Then, for $A=T^{a_1}V^{a_2}\cdots T^{a_{\ell-1}}V^{a_\ell}$ with $\ell \geq 2$ even, $a_i\geq 1$ for all $i=1,\ldots,\ell$ and $w=[\overline{a_1;\ldots,a_\ell}]$, we have
\[\Re(I_f(A))=\int_{\pi/3}^{2\pi/3}f(e^{it})\sin(t)\big(S(w,t)+S(w^\op,t)\big)dt.\]
\end{prop}
\begin{proof}
Lemmas \ref{lem:indentity_cycle_int_BI} and \ref{lem:prop_of_F}$(i)$, together with Remark \ref{rmk:mod_functions_with_real_fourier_coefficients} and \eqref{eq:identity_function_F}, imply that
 \[\Re(I_f(A))=\int_{\pi/3}^{2\pi/3} f(e^{it})\sin(t)\sum_{k=1}^{s(A)} \Big(F(w^{(k)},t)+F(-\tilde w^{(k)},\pi-t)\Big)dt.\]
 Now, recall that each term $w^{(k)}$ is of the form \eqref{eq:w^k}. Using Lemma \ref{lem:prop_of_F}$(ii)$, we can write
\[F(w^{(k)},t)=F([j;\overline{a_{i+1},a_{i+2},\ldots,a_\ell,a_1,\ldots,a_i}],t) \quad \text{with }1\leq j\leq a_i.\]
Hence \eqref{eq:def_of_S(w,t)} gives
\[\sum_{k=1}^{s(A)} F(w^{(k)},t)=S(w,t).\]
Similarly, using  \eqref{eq:tilde_w^k} and Lemma \ref{lem:prop_of_F}$(ii)$, but changing $j$ to $a_i-j$, we get
\[F(-\tilde w^{(k)},\pi-t)=F([j;\overline{a_{i-1},\ldots,a_1,a_\ell,\ldots,a_i}],\pi-t) \quad \text{with }0\leq j\leq a_i-1,\]
thus
\[\sum_{k=1}^{s(A)}F(-\tilde w^{(k)},\pi-t)=\sum_{i=1}^n\sum_{j=0}^{a_i-1}F([j;\overline{a_{i+1},\ldots,a_1,a_\ell,\ldots,a_i}],\pi-t).\]
Then, a simple rearrangement of this sum, together with Lemma \ref{lem:prop_of_F}$(ii)$ and \eqref{def_of_S(-1/tilde w,t)}, give
\[\sum_{k=1}^{s(A)} F(-\tilde w^{(k)},\pi-t)=S(w^\op,\pi-t).\]
Finally, using the change of variables $t\mapsto \pi-t$ and \(f(e^{i(\pi-t)})=f(-1/e^{it})=f(e^{it})\) we obtain
\[\int_{\pi/3}^{2\pi/3} f(e^{it})\sin(t)S(w^\op,\pi-t)=\int_{\pi/3}^{2\pi/3} f(e^{it})\sin(t)S(w^\op,t).\]
This proves the result. 
\end{proof}


\begin{coro}\label{cor:Lambda_f(x)_is_non_negative}
    $\Lambda_f(x)\in [0,\infty)$ for all real $x\in (0,\infty)$.
\end{coro}
\begin{proof}
  Since $F(x,t)>0$ for all $x\in (0,\infty)$ and $t\in [\pi/3,2\pi/3]$, we see that $S(w,t)$ and $S(w^{\op},t)$ are both positive for all $w\in \R$ quadratic irrational with purely continued fraction expansion. By Remark~\ref{rmk:starting_with_T_ending_V} and Proposition \ref{prop:formula_for_Re(I_f(A))} 
  we conclude $\Re(I_f(A))> 0$ for any hyperbolic matrix $A\in \GLN$. 
  Hence, using Lemma \ref{lem:trivial_bound_I_f(A)} and the definition of $\Lambda_f(x)$ we get $\Lambda_f(x)\in [0,\infty)$ for any $x\in (0,\infty)$.
\end{proof}

\subsection{An upper bound for $\Re(I_f(A))$}

Using Proposition \ref{prop:formula_for_Re(I_f(A))} we will prove the following upper bound for $\Re(I_f(A))$.

\begin{prop}\label{prop:bound_Re(I_f(A))}
  Let $f$ be a weakly holomophic modular function for $\SL$ with real Fourier coefficients. Then, for $A=T^{a_1}V^{a_2}\cdots T^{a_{\ell-1}}V^{a_\ell}$ with $\ell\geq 2$ even, $a_1,a_\ell\geq 0$ and $a_2,\ldots,a_{\ell-1}\geq 1$, we have
\[\Re(I_f(A))=\left(\sup_{t\in [\pi/3,2\pi/3]}|f(e^{it})|\right)\cdot O\left(\ell+\sum_{i=1}^\ell\log(a_i)\right)\]
with an absolute implicit constant.
\end{prop}
\begin{proof}
Using Remark \ref{rmk:starting_with_T_ending_V}, we can assume without loss of generality that $a_1\geq 1$ and $a_\ell\geq 1$. Using Proposition \ref{prop:formula_for_Re(I_f(A))} we see that it is enough to prove the estimate
\[S(w,t)=O\left(\ell+\sum_{i=1}^\ell\log(a_i)\right) \quad \text{for }w=[\overline{a_1;a_2,\ldots,a_\ell}],\]
since the sum $S(w^\op,t)$ is of a similar shape. Moreover, from the definition of  $S(w,t)$ in \eqref{eq:def_of_S(w,t)} we see that is is enough to prove
\[\sum_{j=1}^{a_i}F([j;\overline{a_{i+1},\ldots,a_\ell,a_1,\ldots,a_i}],t)=O\big(1+\log(a_i)\big) \quad \text{for }i=1,\ldots,\ell.\]
From the definition of $F(x,t)$ in \eqref{eq:def_of_F(x,t)} is easy to see that
\(F(x,t)=O\left(1/x\right)\) for all \(x\geq 1\) and \(t\in [\pi/3,2\pi/3]\).
Since $[j;\overline{a_{i+1},\ldots,a_\ell,a_1,\ldots,a_i}]\geq j$, we get
\[\sum_{j=1}^{a_i}F([j;\overline{a_{i+1},\ldots,a_\ell,a_1,\ldots,a_i}],t)=O\left(\sum_{j=1}^{a_i}\frac{1}{j}\right)=O(1+\log(a_i)).\]
This completes the proof.
\end{proof}

\section{First properties of $\Lambda_f$}\label{sec:first_properties_of_Lambda}

In this section we apply the results proven in Section \ref{sec:cycle_integrals} to deduce some basic properties of the $f$-Lyapunov exponent. We start with the following proposition.

\begin{prop}\label{prop:Lambda_f_is_zero_on_positive_rationals}
$\Lambda_f(x)=0$ for every rational $x\in (0,\infty)$.
\end{prop}
\begin{proof}
Let $x\in (0,\infty)$ be a rational number. By our choice of the Farey path for a rational number, the sequence $(A_n)_{n\geq 0}$ associated to $x$ is of the form
\[A_n=T^{b_1}V^{b_2}\cdots T^{b_{\ell-1}}V^{k} \quad \text{ for all }n=b_1+\ldots+b_{\ell-1}+k \text{ with }k\geq 0,\]
for certain integers $b_1,\ldots,b_{\ell-1}\geq 0$. Thus, by Proposition \ref{prop:bound_Re(I_f(A))} we have
\[\Re(I_f(A_n))=O\big(\ell+\log(b_1+1)+\log(b_2)+\ldots+\log(b_{\ell-1})+\log(k)\big)\]
for all $n$ as above. Since $\ell$ is fixed and $k\geq n$, we conclude
\[\Lambda_f(x)=\limsup_{n\to \infty}\frac{\Re(I_f(A_n))}{n}=0.\]
This proves the result.
\end{proof}

\begin{remark}\label{rmk:vanishing_almost_everywhere}
Using Proposition \ref{prop:formula_for_Re(I_f(A))} we see that
\[0\leq \Lambda_f(x)\leq \Lambda_1(x)\left(\sup_{t\in [\pi/3,2\pi/3]}|f(e^{it})|\right) \quad \text{for all }x\in (0,\infty),\]
hence $\Lambda_f(x)=0$ whenever $\Lambda_1(x)=0$. In particular, by \cite[Theorem 4]{SV17} we deduce that $\Lambda_f(x)=0$ for almost all $x\in (0,\infty)$. Alternatively, one can show this property for $\Lambda_f$ by replicating the proof given in \cite[Theorem 4]{SV17} for $\Lambda_1$.
\end{remark}


For the proof of Theorem \ref{thm:first_properties_Lambda_f} we need the following lemma where we use the matrix
\(\tilde{S}:=\left(\begin{smallmatrix}
    0 & 1 \\ 1 & 0
\end{smallmatrix}\right)\in \GLN.\)
Recall that conjugation by $\tilde{S}$ has the effect of interchanging $T$ and $V$ (see \eqref{eq:STS=V}). In particular, if $A=T^{a_1}V^{a_2}\cdots V^{a_{\ell}}$ then $\tilde{S}A\tilde{S}=V^{a_1}T^{a_2}\cdots T^{a_{\ell}}$.

\begin{lemma}\label{lem:matrices_of_Mx_1/x_and_1-x}
Let $x\in(0,\infty)$ with $(A_n)_{n\geq 0}$ the associated sequence of matrices in $\SLN$ defined by the path in the Farey tree $\Far$ converging to $x$. Then, the following properties hold:
\begin{enumerate}
\item[$(i)$] For $M\in \{T,V\}$, the sequence $(B_n)_{n\geq 0}$ of matrices associated to $M(x)$ is given by
\begin{equation}\label{eq:matrices_of_Mx}
B_0=I, \quad \quad B_n=MA_{n-1} \quad \text{for }n\geq 1.
\end{equation}
\item[$(ii)$] If $x$ is irrational, then the sequence $(B_n)_{n\geq 0}$ of matrices associated to $\tilde{S}x=\frac{1}{x}$ is given by
\begin{equation}\label{eq:matrices_of_1/x_irrational}
    B_n:=\tilde{S}A_n\tilde{S}.
\end{equation}
\item[$(iii)$] If $x\in (0,1/2)$ and $x$ is irrational, then the sequence $(B_n)_{n\geq 0}$ of matrices associated to $1-x$ is given by
\begin{equation}\label{eq:matrices_of_1-x}
B_0=I, \quad \quad B_n=VT^{-1} \tilde S A_n \tilde S \quad \text{for }n\geq 1.
\end{equation}
\end{enumerate}
\end{lemma}
\begin{proof}
It follows from the uniqueness of the path in $\Far$ converging to $x$ (with the convention that this path eventually becomes a sequence of infinitely many consecutive left turns when $x$ is rational) that the sequence of matrices  $(A_n)_{n\geq 0}$ is uniquely determined by the following three properties:
\begin{itemize}[leftmargin=*]
    \item For every $n\geq 0$, $A_n^{-1}A_{n+1}\in \{T,V\}$.
    \item The sequence $(A_n^{-1}A_{n+1})_{n\geq 0}$ consists of infinitely many $T$'s and $V$'s if $x$ is irrational, or is eventually constant equal to $V$ if $x$ is rational.
    \item $A_n(1)\to x$ as $n\to \infty$.
\end{itemize}
A straightforward computation shows that the sequences \eqref{eq:matrices_of_Mx}, \eqref{eq:matrices_of_1/x_irrational} and \eqref{eq:matrices_of_1-x} satisfy the first two properties above. Now, for the sequence \eqref{eq:matrices_of_Mx} we have
\[\lim_{n\to \infty}B_n(1)=\lim_{n\to \infty}MA_{n-1}(1)=M(x).\]
Similarly,  the sequence \eqref{eq:matrices_of_Mx} satisfies
\[\lim_{n\to \infty}B_n(1)=\lim_{n\to \infty}\tilde {S}A_{n}(1)=\tilde{S}(x)=\frac{1}{x},\]
while the sequence \eqref{eq:matrices_of_1-x} has
\[\lim_{n\to \infty}B_n(1)=\lim_{n\to \infty}\tilde V\tilde S V^{-1}A_n(1)=V\tilde S V^{-1}(x)=1-x,\]
hence these are the sequences associated to $M(x)$ (for $M\in \{T,V\}$), $\frac{1}{x}$ and $1-x$, respectively. This proves the lemma.
\end{proof}


\subsection{Proof of Theorem \ref{thm:first_properties_Lambda_f}}

We are now in position to prove Theorem \ref{thm:first_properties_Lambda_f}. The fact that the limit superior in \eqref{eq:def_lambda_f_intro} is finite is given in Lemma \ref{lem:trivial_bound_I_f(A)}. We extend $\Lambda_f$ to a function $\Lambda_f:\P^1(\R)\to \R$ by defining
\[\Lambda_f(x):=\Lambda_f(-x) \text{ for }x\in (-\infty,0] \text{ and }\Lambda_f(\infty)=0.\]
By Proposition \ref{prop:Lambda_f_is_zero_on_positive_rationals} we conclude that $\Lambda_f(x)=0$ for all $x\in \P^1(\mathbb{Q})$. Now, we have to prove that $\Lambda_f(M(x))=\Lambda_f(x)$ for all $M\in \GL$ and $x\in \P^1(\R)$. This is trivially true on $\P^1(\mathbb{Q})$, so we only need to check this invariance property on irrational numbers $x\in \R$. For such $x$ we have $\Lambda_f(M(x))=\Lambda_f(x)$ for $M=\left(\begin{smallmatrix}
    -1 & 0 \\ 0 & 1
\end{smallmatrix}\right)$ by construction. Since $\GL$ is generated by $\tilde{S}$, $\left(\begin{smallmatrix}
    -1 & 0 \\ 0 & 1
\end{smallmatrix}\right)$ and $T$, we are left with the cases $M\in \{\tilde{S},T\}$. Assume $x\in (0,\infty)$ and let $(A_n)_{n\geq 0}$ and $(B_n)_{n\geq 0}$ denote the sequences of matrices associated to $x$ and $M(x)$. If $M=T$ then by Lemma \ref{lem:matrices_of_Mx_1/x_and_1-x}$(i)$ we have $B_n=TA_{n-1}$ for all $n\geq 1$. Then, by Theorem \ref{thm:almost_invariance_cycle_integrals} we get
\[\Lambda_f(T(x))=\limsup_{n\to \infty} \frac{\Re(I_f(B_n))}{n}=\limsup_{n\to \infty} \left(\frac{\Re(I_f(A_{n-1}))}{n-1}\right)\left(\frac{n-1}{n}\right)=\Lambda_f(x).\]
Similarly, if $M=\tilde{S}$ then by Lemmas \ref{lem:matrices_of_Mx_1/x_and_1-x}$(ii)$ and \ref{lem:elementary_prop_cycle_integrals}$(iv)$ we have
\[\Lambda_f(\tilde{S}(x))=\limsup_{n\to \infty} \frac{\Re(I_f(\tilde{S} A_n \tilde{S}))}{n}=\limsup_{n\to \infty} \frac{\Re(I_f( A_n ))}{n}=\Lambda_f(x).\]
This proves that $\Lambda_f(x)$ is invariant under the action of $T$ and $\tilde{S}$ for $x\in(0,\infty)$. Invariance under $M=\tilde{S}$ extends to $x\in (-\infty,0)$ since $\tilde{S}$ preserves the sign. Similarly, the invariance of $\Lambda_f(x)$ under $M=T$ extends to $x\in (-\infty,-1)$. 
Now, assume $x\in (-1,0)$. Since by definition $\Lambda_f(x)=\Lambda_f(-x)$, putting $y:=-x\in (0,1),$ we need to check that $\Lambda_f(1-y)=\Lambda_f(y)$. Changing $y$ to $1-y$ if necessary, we can assume $y\in (0,1/2)$. If we now denote by $(A_n)_{n\geq 0}$ the sequence of matrices associated to $y$, then the sequence of matrices $(B_n)_{n\geq 0}$ associated to $1-y$ is given by \eqref{eq:matrices_of_1-x} in Lemma \ref{lem:matrices_of_Mx_1/x_and_1-x}$(iii)$. Hence, we have
\[\Lambda_f(1-y)=\limsup_{n\to \infty}\frac{\Re(I_f(VT^{-1} \tilde S A_n \tilde S))}{n}=\limsup_{n\to \infty}\frac{\Re(I_f( A_n ))}{n}=\Lambda_f(y)\]
by Theorem \ref{thm:almost_invariance_cycle_integrals} and Lemma \ref{lem:elementary_prop_cycle_integrals}$(iv)$. This completes the proof of the theorem.

\subsection{Proof of Theorem \ref{thm:Lambda_f_on_quad_irrationals}} 

We now proceed to prove Theorem \ref{thm:Lambda_f_on_quad_irrationals}. First, we recall that for an irrational number $x\in (0,\infty)$ with continued fraction  expansion $x=[a_1; a_2,a_3,\ldots]$, the associated sequence $(A_n)_{n\geq 0}$ is given by \eqref{eq:hyp_matrix_seq_intro}.

Let $x=[a_1;a_2,\ldots,a_r,\overline{b_1,\ldots,b_\ell}]$ be a quadratic irrationality with $\ell\geq 2$ even and put
$y:=[\overline{b_1;\ldots,b_\ell}]$. Since $x$ and $y$  are $\GL$-equivalent, we have $\Lambda_f(x)=\Lambda_f(y)$ by Theorem \ref{thm:first_properties_Lambda_f}. 
Now, let $(B_n)_{n\geq 0}$ be the sequence of matrices associated to $y$, define $B:=T^{b_1}V^{b_2}\cdots T^{b_{\ell-1}}V^{b_\ell}$ and put $m:=s(B)=b_1+\ldots+b_\ell$. It follows from \eqref{eq:hyp_matrix_seq_intro} that for every $n=mq+k$ with $q\geq 0$ and $0\leq k<m$, we have
\(B_n=B^qB_k\). 
Then, by Theorem \ref{thm:almost_additivity_for_badly_approximable} and the fact that $I_f(B^q)=qI_f(B)$ (Lemma \ref{lem:elementary_prop_cycle_integrals}$(ii)$) we get
\[I_f(B_n)=qI_f(B)+I_f(B_k)+\left(\sup_{t\in [\pi/3,2\pi/3]}|f(e^{it})|\right)\cdot O\left(C\right)\]
where $C:=\max\{b_1,\ldots,b_\ell\}$. Since $q=\lfloor \frac{n}{m}\rfloor$, this implies
\[\Lambda_f(y)=\lim_{n\to \infty}\frac{\lfloor \frac{n}{m}\rfloor \Re(I_f(B))}{n}=\frac{\Re(I_f(B))}{m}=\frac{\Re(I_f(T^{b_1}V^{b_2}\cdots T^{b_{\ell-1}}V^{b_\ell}))}{b_1+\ldots+b_\ell},\]
as desired. This completes the proof of Theorem \ref{thm:Lambda_f_on_quad_irrationals}.
 
\section{Values attained by $\Lambda_f$}\label{sec:Values attained by Lambda_f}

This section is devoted to the proof of Theorem \ref{thm:values_of_Lambda_f}. Note that by Theorem \ref{thm:first_properties_Lambda_f} and Corollary \ref{cor:Lambda_f(x)_is_non_negative} we have $\Lambda_f(\P^1(\R))\subseteq [0,\infty)$, while our goal is  to prove that $\Lambda_f(\P^1(\R))=[0,\Lambda_f(\phi)]$. The proof of $\Lambda_f(\P^1(\R))\subseteq [0,\Lambda_f(\phi)]$ reduces to the following key result.

\begin{prop}\label{prop:S(w,t)_leq_golden_ratio}
For every purely periodic quadratic irrationality $w=[\overline{a_1;a_2,\ldots,a_\ell}]$ with $\ell\geq 2$ even, we have
  \[S(w,t)\leq \frac{s(w)}{2}S(\phi,t) \quad \text{for all }t\in [\pi/3,2\pi/3],\]
  where $S(w,t)$ and $s(w)$ are defined in \eqref{eq:def_of_S(w,t)} and \eqref{sn}, respectively. Moreover, the inequality can be made strict if $w\neq \phi$.
\end{prop}

Let \(\Phi:=\left(\begin{smallmatrix}
    1 & 1 \\ 1 & 0
\end{smallmatrix}\right)\) which is the generator  of the stabilizer in $\GLN$ of the golden ratio $\phi$. Note that $\Phi$ acts as $\Phi(x)=1+\frac{1}{x}$ for $x\in \P^1(\R)$.
To give the proof of Proposition \ref{prop:S(w,t)_leq_golden_ratio} we start with the following lemma about the function $F(x,t)$ defined in \eqref{eq:def_of_F(x,t)}. The proof of this lemma reduces to straightforward computations and it is given in the appendix.

\begin{lemma}\label{lem:key_ineq_FT}
   For every $t \in [\pi/3,2\pi/3]$ the function
   \(x\mapsto F(x,t)+F(\Phi(x),t)\)
   is decreasing in $[4/3,\infty)$. In particular:
   \begin{equation}\label{eq:F(x,t)+F(Phix,t)_less_than_2F(phi,t)}
   F(x,t)+F(\Phi(x),t)< 2F(\phi,t) \quad \text{for all }x\in (\phi,\infty).      
   \end{equation}
\end{lemma}

\begin{remark}\label{rmk:4/3}
In our proof of Proposition \ref{prop:S(w,t)_leq_golden_ratio} below we only use Lemma \ref{lem:key_ineq_FT} in the form of \eqref{eq:F(x,t)+F(Phix,t)_less_than_2F(phi,t)}. The stronger statement regarding the monotonicity of $F(x,t)+F(\Phi(x),t)$ will be used later when studying the values of $\Lambda_f$ at Markov irrationalities. The relevance of the constant $4/3$ above is the following: any irrational number with continued fraction expansion $x=[a_1;a_2,a_3,\ldots]$ with $a_i\in \{1,2\}$ for all $i$ satisfies $x\in [4/3,\infty)$. Indeed, since $a_3\geq 1$, $a_2\leq 2$ and $a_1\geq 1$ we have \(x\geq 1+\frac{1}{2+\frac{1}{1}}=\frac{4}{3}\). 
This fact will be used to estimate the value of $F(x,t)$ for $x$ in the cycle of a Markov irrationality.
\end{remark}

For the proof of Proposition \ref{prop:S(w,t)_leq_golden_ratio} we also need the following well known result which can be found, e.g., in \cite[Lemma 1.24]{Aig2013}.

\begin{lemma}\label{lem:comparison_continued_fractions}
    Let $u=[a_1;a_2,\ldots]$, $v=[b_1;b_2,\ldots]$ be two different real numbers. Then $u<v$ if and only if $(-1)^{i+1}a_i<(-1)^{i+1}b_i$ where $i\geq 1$ is the first index for which $a_i\neq b_i$.
\end{lemma}

Recall that, given a quadratic irrationality $w=[\overline{a_1;a_2,\ldots,a_\ell}]$ with $\ell\geq 2$ even and minimal, the sum $S(w,t)$ for $t\in [\pi/3,2\pi/3]$ is defined in \eqref{eq:def_of_S(w,t)}. It can be rewritten  as
\begin{equation}\label{eq:def_of_S_i(w,t)}
  S(w,t)=\sum_{i=1}^{\ell}S_i(w,t) \quad \text{with }  S_i(w,t):=\sum_{j=1}^{a_i}F([j;\overline{a_{i+1},\ldots,a_\ell,a_1,\ldots,a_i}],t).
\end{equation}
We extend the definition of $S_i(w,t)$ to $i\in \Z$ by putting $S_i(w,t):=S_{i_0}(w,t)$ where $i_0$ is the unique integer satisfying $1\leq i_0\leq \ell$ and $i\equiv i_0$ mod $\ell$.

\begin{proof}[Proof of Proposition \ref{prop:S(w,t)_leq_golden_ratio}]
    Clearly we can assume $w\neq \phi$. Fix $t\in [\pi/3,2\pi/3]$. We introduce the following notation: a term with $1\leq j\leq a_i$ in $S_i(w,t)$ is called \textit{bad} if 
\[F([j;\overline{a_{i+1},\ldots, a_\ell,a_1,\ldots,a_{i}}],t) >F(\phi,t),\]
and it is called \textit{good} otherwise. Using  the fact that $x\mapsto F(x,t)$ is decreasing for $x\in[1,\infty)$ (Lemma \ref{lem:prop_of_F}), together with $\phi=[\overline{1;1}]\in (1,2)$, we see that the bad terms are exactly those with $[j;\overline{a_{i+1},\ldots, a_\ell,a_1,\ldots,a_{i}}]\in (1,\phi)$. By  Lemma \ref{lem:comparison_continued_fractions} this happens only when $j=1$ and the first partial quotient greater than 1 in the periodic part $\overline{a_{i+1},\ldots,a_\ell,a_1,\ldots,a_i}$ is located in an odd position when counting from left to right and starting with $a_{i+1}$ (ignoring their indices). In particular, the $i$-th sum $S_i(w,t)$ in \eqref{eq:def_of_S_i(w,t)} has at most one bad term. 

We will prove that if we have a bad term in $S_i(t,w)$, then we can pair it with a good term in $S_{i+1}(w,t)$ in such a way that we can apply \eqref{eq:F(x,t)+F(Phix,t)_less_than_2F(phi,t)}.  So, assume we have a bad term $F([1;\overline{a_{i+1},\ldots, a_\ell,a_1,\ldots,a_{i}}],t)$. 
Then, there is an odd integer $k\geq 1$ such that $a_{i+1}=\ldots =a_{i+k-1}=1$ and $a_{i+k}>1$  (the $k$-th coefficient in the periodic part $\overline{a_{i+1},\ldots, a_\ell,a_1,\ldots,a_{i}}$ is the first partial quotient greater than 1). We claim that we can pair this bad term with the good term
\[F([a_{i+1};\overline{a_{i+2},\ldots,a_\ell,a_1,\ldots ,a_i,a_{i+1}}],t)=F([\overline{a_{i+1};\ldots,a_{i}}],t)\]
appearing in $S_{i+1}(w,t)$. Indeed, put $x:=[\overline{a_{i+1};\ldots,a_{i}}]$ and note that $x>\phi$ (this follows from Lemma \ref{lem:comparison_continued_fractions} and the fact that $k$ is odd). By monotonicity, $F(x,t)<F(\phi,t)$ and hence $F(x,t)$ is a good term. Since
\[[1;\overline{a_{i+1},\ldots, a_\ell,a_1,\ldots,a_{i}}]=1+\frac{1}{x}=\Phi(x),\]
we can apply \eqref{eq:F(x,t)+F(Phix,t)_less_than_2F(phi,t)} and get $F(x,t)+F(\Phi(x),t)< 2F(\phi,t)$.

After matching all the possible bad terms in $S(w,t)$ with the corresponding good terms as above, we conclude \(S(w,t)< s(w)F(\phi,t)=\frac{s(w)}{2}S(\phi,t)\). This proves the proposition.
\end{proof}

In the proof that every value in $[0,\Lambda_f(\phi)]$ is attained by $\Lambda_f$, we use the following lemma which is a direct consequence of Theorem \ref{thm:almost_additivity_for_badly_approximable} and Lemma \ref{lem:elementary_prop_cycle_integrals}$(ii)$.

\begin{lemma}\label{lem:I_f(AB^m)_goes_to_I_f(B)}
    Let $A=T^{a_1}V^{a_2}\ldots T^{a_{r-1}}V^{a_r}$ and $B=T^{b_1}V^{b_2}\ldots T^{b_{s-1}}V^{b_s}$ with $r,s\geq 2$ both even and $a_i,b_j\geq 1$ for all $i,j$. Then
    \[\lim_{m\to \infty}\frac{I_f(AB^m)}{s(A)+ms(B)}=\frac{I_f(B)}{s(B)}.
    \]
\end{lemma}

\subsection{Proof of Theorem \ref{thm:values_of_Lambda_f}}


Let $x$ be an irrational number with associated sequence of hyperbolic matrices $(A_n)_{n\geq 0}$. The upper bound $\Lambda_f(x)\leq \Lambda_f(\phi) $ 
is a consequence of the inequality $\Re(I_f(A_n))\leq n \Lambda_f(\phi)$ for all $n\geq 0$. 
This in return,  follows from Theorem \ref{thm:Lambda_f_on_quad_irrationals} and Propositions~\ref{prop:formula_for_Re(I_f(A))} and~\ref{prop:S(w,t)_leq_golden_ratio}.

We will now prove that any value $\lambda_0\in [0,\Lambda_f(\phi)]$ is attained by $\Lambda_f$. Since $0=\Lambda_f(1)$, we can assume $\lambda_0\in (0,\Lambda_f(\phi))$. We will construct a sequence of matrices $(A_n)_{n\geq 0}$ with $A_0=I$ and $A_n^{-1}A_{n+1}\in \{T,V\}$ for all $n\geq 1$, for which the sequence $\frac{\Re(I_f(A_n))}{n}$ takes values in $(0,\lambda_0)$ and in $(\lambda_0,\Lambda_f(\phi))$ infinitely many times, in such a way that
\[\limsup_{n\to \infty}\frac{\Re(I_f(A_n))}{n}=\lambda_0.\]
We start with $A_0=I$. It follows from Proposition \ref{prop:Lambda_f_is_zero_on_positive_rationals} that
\[\lim_{n\to \infty}\frac{\Re(I_f(TV^{n-1}))}{n}=\Lambda_f(1)=0,\]
hence we can choose $a\geq 2$ such that
\begin{equation}\label{eq:choice of a}
\frac{\Re(I_f(TV^{a}))}{a+1} < \lambda_0.
\end{equation}
We then define
\(A_1=T, A_2=TV,A_3=TV^2, \ldots ,A_{1+a}=TV^a.\)
Now, by Lemma \ref{lem:I_f(AB^m)_goes_to_I_f(B)} we have
\[\lim_{m\to \infty}\frac{\Re(I_f(A_{1+a}(TV)^m))}{(1+a)+2m}=\frac{\Re(I_f(TV))}{2}=\Lambda_f(\phi)>\lambda_0,\]
hence there exists a minimal $m_1\geq 1$ such that
\begin{equation*}
\frac{\Re(I_f(A_{1+a}(TV)^{m_1}))}{(1+a)+2m_1}>\lambda_0.
\end{equation*}
We then define
\(A_{2+a}=TV^aT, A_{3+a}=TV^aTV, \ldots ,A_{(1+a)+2m_1}=TV^a(TV)^{m_1}.\)
Now, we use Lemma \ref{lem:I_f(AB^m)_goes_to_I_f(B)}  to deduce that
\[\lim_{m\to \infty}\frac{\Re(I_f(A_{(1+a)+2m_1}(TV^a)^m))}{(1+a)+2m_1+(1+a)m}=\frac{\Re(I_f(TV^a))}{1+a}<\lambda_0,\]
hence there exists a minimal $m_2\geq 1$ such that
\begin{equation*}
\frac{\Re(I_f(A_{(1+a)+2m_1}(TV^a)^{m_2}))}{(1+a)+2m_1+(1+a)m_2}<\lambda_0,   
\end{equation*}
and we put
\(A_{(1+a)+2m_1+1}=TV^a(TV)^{m_1}T,A_{(1+a)+2m_1+2}=TV^a(TV)^{m_1}TV, \ldots\)
until we reach $A_{(1+a)+2m_1+(1+a)m_2}=TV^a(TV)^{m_1}(TV^a)^{m_2}$. Repeating this process  we obtain a sequence of matrices $(A_n)_{n\geq 0}$ satisfying the following property: there is a sequence of positive integers $(m_i)_{i\geq 1}$ such that for all $i\geq 2$ even we have
\begin{equation*}
    \frac{\Re(I_f(A_{\ell_i}))}{\ell_i}<\lambda_0\leq \frac{\Re(I_f(A_{\ell_i-(1+a)}))}{\ell_i-(1+a)}
\end{equation*}
for $\ell_i:=(1+a)+2m_1+(1+a)m_2+\ldots+2m_{i-1}+(1+a)m_i$, while for all $i\geq 3$ odd we have
\begin{equation}\label{eq:r_i}
\frac{\Re(I_f(A_{r_i-2}))}{r_i-2}\leq \lambda_0<\frac{\Re(I_f(A_{r_i}))}{r_i}
\end{equation}
for $r_i:=(1+a)+2m_1+(1+a)m_2+\ldots+(1+a)m_{i-1}+2m_i$. Indeed, these inequalities follow from the minimality of the $m_i$'s.

The sequence $(A_n)_{n\geq 0}$ defines a path in the Farey tree $\Far$ converging to a badly approximable real number $x_0$  of the form
\(x_0=[1;a,1,a,\ldots,1,a,1,1,\ldots ,1,1,1,a,1,a,\ldots]\).
We claim that
\begin{equation}\label{eq:Lambda(x_0)=lambda_0}
\Lambda_f(x_0)=\limsup_{n\to \infty}\frac{\Re(I_f(A_n))}{n}=\lambda_0.    
\end{equation}
Since by construction there are infinitely many indices $n$ for which $\frac{\Re(I_f(A_n))}{n}>\lambda_0$, namely for $n=r_i$ with $i$ odd, we have
\(\Lambda_f(x_0)\geq \lambda_0.\)
Now, let $\epsilon>0$. We will prove that for all large enough $n$  the inequality
\begin{equation}\label{eq:Lambda(x_0)<lambda_0+3eps}
\frac{\Re(I_f(A_n))}{n}\leq \lambda_0+3\epsilon 
\end{equation}
holds. This will imply that $\Lambda_f(x_0)\leq \lambda_0+3\epsilon$ which in turn completes the proof of \eqref{eq:Lambda(x_0)=lambda_0}. 

In order to prove \eqref{eq:Lambda(x_0)<lambda_0+3eps} observe that by Corollary \ref{cor:I(A_m+n)-I(A_n)_goes_to_zero} there exists $N_1\in \N$ such that
\begin{equation}\label{eq:I_f(A_n+m)-I_f(A_n)}
    \left|\frac{\Re(I_f(A_{n+m}))}{n+m}-\frac{\Re(I_f(A_{n}))}{n}\right|\leq \epsilon \quad \text{for all }n\geq N_1 \text{ and }0\leq m\leq 1+a.
\end{equation}
In particular, choosing $m=2$ and $i$ odd and large enough  so that $n=r_i-2>N_1$, we  deduce from \eqref{eq:r_i}  that
\begin{equation}\label{eq:ineq_ri}
    \frac{\Re(I_f(A_{r_i}))}{r_i}\leq \lambda_0+\epsilon .
\end{equation}
By Theorem \ref{thm:almost_additivity_for_badly_approximable} and Lemma \ref{lem:elementary_prop_cycle_integrals}$(iii)$, for all $k\geq 0$ and $i$ odd and large enough, we have 
\[\frac{\Re(I_f(A_{r_i}(TV^a)^k)}{r_i+k(1+a)}=\frac{\Re(I_f(A_{r_i}))}{r_i}\frac{r_i}{r_i+k(1+a)}+\frac{\Re(I_f(TV^a))}{1+a}\frac{k(1+a)}{r_i+k(1+a)}+\epsilon.\]
In particular, 
using \eqref{eq:ineq_ri} and \eqref{eq:choice of a} we get
\begin{equation}\label{eq:A_r_i(TV^a)^k)}
    \frac{\Re(I_f(A_{r_i}(TV^a)^k)}{r_i+k(1+a)}\leq \left(\lambda_0+\epsilon\right)\frac{r_i}{r_i+k(1+a)}+\lambda_0 \frac{k(1+a)}{r_i+k(1+a)}+\epsilon\leq \lambda_0+2\epsilon.
\end{equation}
Denote by $i_0$ a large enough odd integer satisfying $r_{i_0}\geq N_1$ and let $n\geq r_{i_0}$. Since $r_1<\ell_2<r_3<\ell_4<\ldots$, there is a unique index $j \geq i_0$ such that
\begin{itemize}[leftmargin=*]
    \item $r_j\leq n<\ell_{j+1}=r_j+(1+a)m_{j+1}$ and $j$ is odd, or
    \item $\ell_j\leq n<r_{j+1}=\ell_j+2m_{j+1}$ and $j$ is even.
\end{itemize}
In the first case, we write $n=r_j+(1+a)k+m$ with $0\leq k<m_{j+1}$ and $0\leq m<1+a$. Then, combining \eqref{eq:I_f(A_n+m)-I_f(A_n)} and \eqref{eq:A_r_i(TV^a)^k)} we get
\[\frac{\Re(I_f(A_{n}))}{n}\leq \frac{\Re(I_f(A_{r_j+(1+a)k}))}{r_j+(1+a)k}+\epsilon=\frac{\Re(I_f(A_{r_i}(TV^a)^k)}{r_i+k(1+a)}+\epsilon\leq \lambda_0+3\epsilon.\]
In the second case, we write $n=\ell_j+2k+m$ with $0\leq k<m_{j+1}$ and $0\leq m<2$. Then, using \eqref{eq:I_f(A_n+m)-I_f(A_n)} and the minimality of $m_{j+1}$ we get
\[\frac{\Re(I_f(A_{n}))}{n}\leq \frac{\Re(I_f(A_{\ell_j+2k}))}{\ell_j+2k}+\epsilon\leq \lambda_0+\epsilon.\]
This proves \eqref{eq:Lambda(x_0)<lambda_0+3eps} in both cases and completes the proof of Theorem \ref{thm:values_of_Lambda_f}.

\begin{remark}\label{rmk:badly_approx_have_Lambda>0_and_golden_ratio_strict_maximum}
    It follows from our proof of Theorem \ref{thm:values_of_Lambda_f} that in fact all values in $(0,\Lambda_f(\phi)]$ are attained by $\Lambda_f$ on badly approximable numbers. Moreover, the following converse result holds:  every badly approximable number $x\in \R$ has $\Lambda_f(x)>0$. Indeed, using Theorem \ref{thm:first_properties_Lambda_f}, Proposition \ref{prop:formula_for_Re(I_f(A))} and the fact that $u\mapsto F(u,t)$ is decreasing for $u\in [1,\infty)$ for every fixed $t\in [\pi/3,2\pi/3]$ (Lemma \ref{lem:prop_of_F}), we have that every badly approximable $x=[a_1;a_2,a_3,\ldots]$ with $C:=\max\{a_i:i\geq 2\}$ has $\Lambda_f(x)$ bounded below by
    \[ 2\int_{\pi/3}^{2\pi/3}f(e^{it})\sin(t)F(C+1,t)dt\geq 2\left(\frac{C+1}{C^2+3C+3}\right)\int_{\pi/3}^{2\pi/3}f(e^{it})\sin(t)dt>0,\]
where in the second inequality we used that $F(x,t)\geq \frac{x}{1+x+x^2}.$
\end{remark}

\section{Minimum over Markov irrationalities}\label{sec:Minimum over Markov irrationalities}

The goal of this section is to prove Theorem \ref{thm:Lambda_f(M)}, which states that $\Lambda_f(\M)\subset [\Lambda_f(\psi),\Lambda_f(\phi)]$, where $\psi=1+\sqrt{2}=[\overline{2;2}]$ is the silver ratio. Since $\Lambda_f(\M)\subseteq \Lambda_f(\P^1(\R))= [0,\Lambda_f(\phi)]$ by Theorem \ref{thm:values_of_Lambda_f}, we are left to prove the lower bound $\Lambda_f(w)\geq \Lambda_f(\psi)$ for all $w\in \M$. The full proof is given at the end of this section. It reduces to the following key proposition.

\begin{prop}\label{prop:S(w,t)_geq_silver_ratio}
For every Markov irrationality $w\in \M$ we have
\[S(w,t),S(w^{\op},t)\geq \frac{s(w)}{4}S(\psi,t) \quad \text{for all }t\in [\pi/3,2\pi/3].\]
Moreover, the inequality can be made strict if $w\neq \psi$.
\end{prop}
\begin{remark}
        Note that, since $\psi-1=[1;\overline{2}]=\Phi(\psi)$, we have
\begin{equation}\label{eq:S(psi,t)}
   S(\psi,t)=S(\psi^{\op},t)=2(F(\psi,t)+F(\Phi(\psi),t)). 
\end{equation}
\end{remark}

Proposition \ref{prop:S(w,t)_geq_silver_ratio} is the analogue of Proposition \ref{prop:S(w,t)_leq_golden_ratio} in Section \ref{sec:Values attained by Lambda_f} in terms of comparing Markov irrationalities with the silver ratio instead of the golden ratio. As in the case of Proposition \ref{prop:S(w,t)_leq_golden_ratio}, we need to first prove a special property of the function $F(x,t)$. To this purpose, we define \(\Psi:=\left(\begin{smallmatrix}
    2 & 1 \\ 1 & 0
\end{smallmatrix}\right)\)
which is the generator  of the stabilizer in $\GLN$ of the silver ratio $\psi$. 

\begin{lemma}\label{lem:key_ineq_FT2}
   For every $t \in [\pi/3,2\pi/3]$ the function
   \(x\mapsto F(x,t)+F(\Phi(x),t)+F(\Psi(x),t)\)
   is decreasing for $x\in[\phi,\infty)$ and the function
\(x\mapsto F(\Phi\circ \Psi(x),t)\)
   is decreasing for $x\in (0,\infty)$.
\end{lemma}

The proof of Lemma \ref{lem:key_ineq_FT2} is given in the appendix.

\subsection{Proof of Proposition \ref{prop:S(w,t)_geq_silver_ratio}}

We can assume $w\neq \psi$. 
We  start by collecting the purely periodic terms in $S(w,t)$ in \eqref{eq:def_of_S(w,t)} to write $S(w,t)=S_{\text{pp}}(w,t)+S_{\text{np}}(w,t)$ where
\begin{eqnarray*}
S_{\text{pp}}(w,t)&:=&\sum_{1\leq i\leq \ell}F([\overline{a_i;a_{i+1},\ldots, a_\ell,a_1,\ldots,a_{i-1}}],t),\\
S_{\text{np}}(w,t)&:=&\sum_{\substack{1\leq i\leq \ell\\ a_i=2}}F([1;\overline{a_{i+1},\ldots, a_\ell,a_1,\ldots,a_{i}}],t).
\end{eqnarray*}
Let $a_0:=a_\ell$ and $a_{-1}:=a_{\ell-1}$. The sum $S_{\text{pp}}(w,t)$ equals
\begin{eqnarray*}
\sum_{\substack{1\leq i\leq \ell\\ a_{i-1}=2}}F([\overline{a_i;a_{i+1},\ldots, a_\ell,a_1,\ldots,a_{i-1}}],t)+\sum_{\substack{1\leq i\leq \ell\\ a_{i-1}=1}}F([\overline{a_i;a_{i+1},\ldots, a_\ell,a_1,\ldots,a_{i-1}}],t).
\end{eqnarray*}
Letting $i=j-1$, and using that for $j$ odd $a_{j-2}=a_{j-1}$, we rewrite 
\begin{eqnarray*}
\sum_{\substack{1\leq j\leq \ell\\ a_{j-2}=2,j \text{ odd}}}F([\overline{a_{j-1};a_{j},\ldots, a_\ell,a_1,\ldots,a_{j-2}}],t) = \sum_{\substack{1\leq j\leq \ell\\ a_{j-1}=2,j \text{ odd}}}F([\overline{a_{j-1};a_{j},\ldots, a_\ell,a_1,\ldots,a_{j-2}}],t).
\end{eqnarray*}
This implies that $S_{\text{pp}}(w,t)$ equals
\begin{eqnarray*}
&& \sum_{\substack{1\leq i\leq \ell\\ a_{i-1}=2, i \text{ odd}}}\bigg(F([\overline{a_i;a_{i+1},\ldots, a_\ell,a_1,\ldots,a_{i-1}}],t)+ F([\overline{a_{i-1};a_{i},\ldots, a_\ell,a_1,\ldots,a_{i-2}}],t)\bigg)\\
& & +\sum_{\substack{1\leq i\leq \ell\\ a_{i-1}=1}}F([\overline{a_i;a_{i+1},\ldots, a_\ell,a_1,\ldots,a_{i-1}}],t).
\end{eqnarray*}
Now, note that when $a_{i-1}=2$ we have \(\Psi([\overline{a_i;a_{i+1},\ldots,a_{i-1}}])=[\overline{a_{i-1};a_{i},\ldots,a_{i-2}}]\). This shows that $S_{\text{pp}}(w,t)$ equals
\begin{eqnarray*}
&& \sum_{\substack{1\leq i\leq \ell\\ a_{i-1}=2, i \text{ odd}}}\bigg(F([\overline{a_i;a_{i+1},\ldots, a_\ell,a_1,\ldots,a_{i-1}}],t)+F(\Psi([\overline{a_i;a_{i+1},\ldots, a_\ell,a_1,\ldots,a_{i-1}}]),t)\bigg)\\
& & +\sum_{\substack{1\leq i\leq \ell\\ a_{i-1}=1}}F([\overline{a_i;a_{i+1},\ldots, a_\ell,a_1,\ldots,a_{i-1}}],t).
\end{eqnarray*}
Similarly, $S_{\text{np}}(w,t)$ equals
\begin{eqnarray*}
& & \sum_{\substack{1\leq i\leq \ell\\ a_{i-1}=2}}F([1;\overline{a_{i},\ldots, a_\ell,a_1,\ldots,a_{i-1}}],t)\\
&=&  \sum_{\substack{1\leq i\leq \ell\\ a_{i-1}=2,i \text{ odd}}}\bigg(F([1;\overline{a_{i},\ldots, a_\ell,a_1,\ldots,a_{i-1}}],t)+F([1;\overline{a_{i-1},\ldots, a_\ell,a_1,\ldots,a_{i-2}}],t)\bigg).
\end{eqnarray*}
We conclude that $S(w,t)=S_{\text{pp}}(w,t)+S_{\text{np}}(w,t)$ equals
\begin{eqnarray*}
&& \sum_{\substack{1\leq i\leq \ell\\ a_{i-1}=2,i \text{ odd}}}\bigg(F([\overline{a_i;a_{i+1},\ldots, a_\ell,a_1,\ldots,a_{i-1}}],t)+F(\Psi([\overline{a_i;a_{i+1},\ldots, a_\ell,a_1,\ldots,a_{i-1}}]),t)\\
& & \quad +F(\Phi([\overline{a_i;a_{i+1},\ldots, a_\ell,a_1,\ldots,a_{i-1}}]),t)+F(\Phi\circ \Psi([\overline{a_i;a_{i+1},\ldots, a_\ell,a_1,\ldots,a_{i-1}}]),t) \bigg)\\
& & +\sum_{\substack{1\leq i\leq \ell\\ a_{i-1}=1}}F([\overline{a_i;a_{i+1},\ldots, a_\ell,a_1,\ldots,a_{i-1}}],t).
\end{eqnarray*}
Now, for all~$i$ odd the first appearance of a 1, from left to right, among the partial quotients of $ [\overline{a_i;a_{i+1},\ldots, a_\ell,a_1,\ldots,a_{i-1}}]$ occurs in an odd position, and the same holds for the first appearance of a 2. Hence, by Lemma \ref{lem:comparison_continued_fractions}, we have
\begin{equation}\label{eq: compare irrational with phi and psi}
\phi=[\overline{1;1}]\leq [\overline{a_i;a_{i+1},\ldots, a_\ell,a_1,\ldots,a_{i-1}}]< [\overline{2;2}]=\psi,
\end{equation}
thus by Lemma \ref{lem:key_ineq_FT2}  we get
\begin{eqnarray*}
S(w,t)&>& \sum_{\substack{1\leq i\leq \ell\\ a_{i-1}=2,i \text{ odd}}}\bigg(F(\psi,t)+F(\Psi(\psi),t)+F(\Phi(\psi),t)+F(\Phi\circ \Psi(\psi),t) \bigg)\\
& & +\sum_{\substack{1\leq i\leq \ell\\ a_{i-1}=1}}F([\overline{a_i;a_{i+1},\ldots, a_\ell,a_1,\ldots,a_{i-1}}],t).
\end{eqnarray*}
Setting $X:=\{i \in \{1,\ldots,\ell\}:a_{i-1}=2 \text{ and $i$ is odd}\}$ and using \eqref{eq:S(psi,t)} we conclude
\begin{eqnarray*}
S(w,t)&>& \#X \cdot S(\psi,t)+\sum_{\substack{1\leq i\leq \ell\\ a_{i-1}=1}}F([\overline{a_i;a_{i+1},\ldots, a_\ell,a_1,\ldots,a_{i-1}}],t).
\end{eqnarray*} 

Now, define~$Y:=\{i \in \{1,\ldots,\ell\}:a_{i-1}=1 \text{ and $i$ is odd}\}$. Then, arguing as above,  and using 
      that when $a_{i-1}=1$ we have \(\Phi([\overline{a_i;a_{i+1},\ldots,a_{i-1}}])=[\overline{a_{i-1};a_{i},\ldots, a_{i-2}}]\), we obtain
\begin{eqnarray*}
  & &   \sum_{\substack{1\leq i\leq \ell\\ a_{i-1}=1}}F([\overline{a_i;a_{i+1},\ldots, a_\ell,a_1,\ldots,a_{i-1}}],t) \\
    &=& \sum_{i\in Y}\bigg( F([\overline{a_i;a_{i+1},\ldots, a_\ell,a_1,\ldots,a_{i-1}}],t)+F([\overline{a_{i-1};a_i,\ldots, a_\ell,a_1,\ldots,a_{i-2}}],t)\bigg)\\
    &=& \sum_{i\in Y}\bigg( F([\overline{a_i;a_{i+1},\ldots, a_\ell,a_1,\ldots,a_{i-1}}],t)+F(\Phi([\overline{a_i;a_{i+1},\ldots, a_\ell,a_1,\ldots,a_{i-1}}]),t)\bigg).
\end{eqnarray*}
For all $i\in Y$ we have \eqref{eq: compare irrational with phi and psi} hence by Lemma \ref{lem:key_ineq_FT} and \eqref{eq:S(psi,t)} we conclude
\[ \sum_{\substack{1\leq i\leq \ell\\ a_{i-1}=1}}F([\overline{a_i;a_{i+1},\ldots, a_\ell,a_1,\ldots,a_{i-1}}],t) >
     \sum_{i\in Y}\bigg( F(\psi,t)+F(\Phi(\psi),t)\bigg) = \frac{\#Y}{2} \cdot S(\psi,t).\]
We then obtain
\( S(w,t)> \left(\#X+\frac{1}{2}\# Y\right) S(\psi,t)\). 
But since the continued fraction expansion of $x$ is composed of pairs of 2's and 1's, we see that $\# X=\frac{1}{2}\#\{i\in \{1,\ldots,\ell\}: a_i=2\}$ and $\# Y=\frac{1}{2}\#\{i\in \{1,\ldots,\ell\}: a_i=1\}$, and this implies $\#X+\frac{1}{2}\# Y=\frac{s(w)}{4}$. This completes the proof of Proposition \ref{prop:S(w,t)_geq_silver_ratio}.

\subsection{Proof of Theorem \ref{thm:Lambda_f(M)}} As mentioned at the beginning of this section, the proof of Theorem \ref{thm:Lambda_f(M)} reduces to showing that $\Lambda_f(w)\geq \Lambda_f(\psi)$ for all $w\in \M$. This, in turn, follows directly from Proposition \ref{prop:S(w,t)_geq_silver_ratio} together with Proposition \ref{prop:formula_for_Re(I_f(A))} and Theorem \ref{thm:Lambda_f_on_quad_irrationals}.

\begin{remark}\label{rmk:silver_ratio_strict_minimum}
    It follows from Theorem \ref{thm:Lambda_f_on_quad_irrationals} and Proposition \ref{prop:S(w,t)_geq_silver_ratio} that actually we have $\Lambda_f(w)>\Lambda_f(\psi)$ for all $w\in \M\setminus \{\psi\}$.
\end{remark}

\section{Convexity, monotonicity and continuous extension of $\tilde{\Lambda}_f$}\label{sec:convexity_monotonicity_extension}

The purpose of this section is to prove Theorem \ref{thm:Lambda_f_convexity}. We first  reduce the proof  to a convexity property for $\Lambda_f$ on the Markov tree by using  general results about convex and monotonic functions which are collected in the appendix. Then, we will relate the latter to a triangle inequality satisfied by the sums $S(w,t)$ and $S(w^\op,t)$ for $w\in \M$. The proof of Theorem \ref{thm:Lambda_f_convexity} is then given at the end of this section.

\subsection{A convexity criteria for $\tilde{\Lambda}_f$}\label{sec: convexity criteria}

In this section we give a criteria for  the convexity of $\tilde{\Lambda}_f:[0,1/2]\cap \mathbb{Q}\to \R$ in terms of a related convexity   property for the restriction of $\Lambda_f$ to Markov irrationalities. First, recall from the introduction that we parametrize Markov irrationalities by Farey fractions in $\Fartwo$ (see Figure \ref{fig:Farey parametrization}).

\begin{figure}[h!]
\centering
\begin{forest}
[,phantom [$\frac{0}{1}$,name=p1] [] [] []  
[
[$\frac{1}{3}$,name=p2, no edge,tikz={\draw (p2.north)--(p1.south); \draw [thick, <->] (3.5,-2.5)--(4.5,-2.5);} 
[
$\frac{1}{4}$ 
[$\ldots$]
[$\ldots$]]
[
$\frac{2}{5}$
[$\ldots$]
[$\ldots$]
]
]]
[] [] []  [$\frac{1}{2}$,name=p3] tikz={\draw (p2.north)--(p3.south);}
]
\end{forest}
\hspace{0.5cm}
\begin{forest}
[,phantom [{$[\overline{1_2}]$},name=p1] [] [] []  
[
[{$[\overline{2_2,1_2}]$},name=p2, no edge,tikz={\draw (p2.north)--(p1.south);}
[
{$[\overline{2_2,1_4}]$}
[$\ldots$]
[$\ldots$]]
[
{$[\overline{2_4,1_2}]$}
[$\ldots$]
[$\ldots$]
]
]]
[] [] []  [{$[\overline{2_2}]$},name=p3] tikz={\draw (p2.north)--(p3.south);}
]
\end{forest}
\caption{Farey parametrization of Markov irrationalities.}\label{fig:Farey parametrization}
\end{figure}
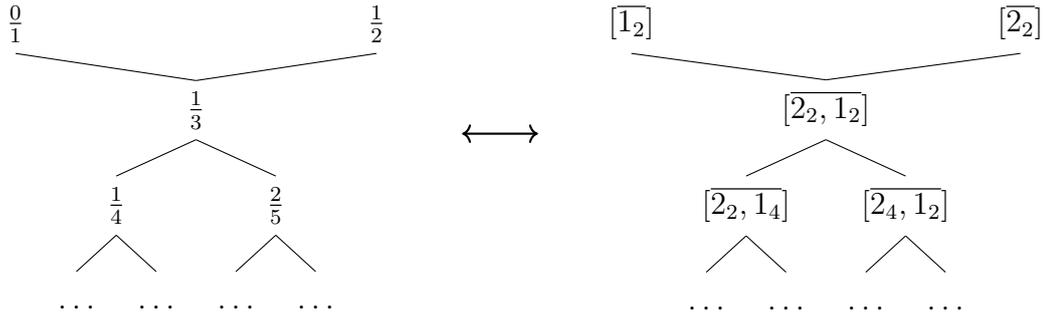

We now define
\[ \Fartwo^{(0)}  :=  \left\{ \frac{0}{1}, \frac{1}{2} \right\}, \quad  \Fartwo^{(1)} := \left\{ \frac{0}{1}, \frac{1}{3}, \frac{1}{2} \right\},  \quad \Fartwo^{(2)}  :=  \left\{ \frac{0}{1}, \frac{1}{4}, \frac{1}{3},\frac{2}{5}, \frac{1}{2} \right\}, \quad \ldots\]
so that, for every integer $n\geq 0$, $\Fartwo^{(n)}$ is the collection of all Farey fractions that appear in the Farey tree $\Fartwo$ up to $n$ levels down on this tree. Every rational number $p/q\in[0,1/2]$ appears in $\Fartwo^{(n)}$ for the first time for some  $n\geq 0$ and we call such $n$ the \emph{minimal level} of $p/q$ in $\Fartwo$. For instance, $1/2$ has minimal level $0$ while $2/5$ has minimal level $2$. We call $\Fartwo^{(n)}$ the \emph{$n$-th level of the Farey tree $\Fartwo$}. 
Accordingly, we define the \emph{$n$-th level $\M^{(n)}$ in the Markov tree} as the set of Markov irrationalities parametrized by Farey fractions in $\Fartwo^{(n)}$. For instance
\[ \M^{(0)}  =  \left\{ [\overline{1_2}],[\overline{2_2}] \right\}, \M^{(1)}  =  \left\{ [\overline{1_2}],[\overline{2_2,1_2}],[\overline{2_2}] \right\}, \M^{(2)} =  \left\{[\overline{1_2}],[\overline{2_2,1_4}],[\overline{2_2,1_2}],[\overline{2_4,1_2}],[\overline{2_2}]\right\}.\]
As in the case of Farey fractions, we call two Markov irrationalities, $w_1$ and $w_2,$ \emph{neighbors} if they appear as consecutive elements in some level $\M^{(n)}$. Equivalently, $w_1, w_2$ are neighbors if the corresponding Farey fractions $\frac{p_1}{q_1}$ and $\frac{p_2}{q_2}$ are neighbors in the Farey tree as defined in the introduction (i.e. $p_1q_2-p_2q_1=-1$). For instance, $[\overline{1_2}]$ and $[\overline{2_2,1_2}]$ are neighbors, but $[\overline{1_2}]$ and $[\overline{2_4,1_2}]$ are not.

Note that, if $\frac{a}{b}<\frac{c}{d}$ are consecutive Farey fractions in $\Fartwo^{(n)}$, then for every integer $m\geq 0$ the Farey fractions
\begin{eqnarray*}
    \frac{a}{b}\oplus \left(\frac{c}{d}\right)^{\oplus m} &:=&\left(\left(\cdots \left(\frac{a}{b} \oplus \frac{c}{d}\right)\oplus \frac{c}{d} \oplus \cdots \right)\oplus \frac{c}{d}\right)\oplus \frac{c}{d},\\ 
        \left(\frac{a}{b}\right)^{\oplus m}\oplus \frac{c}{d} &:=&\frac{a}{b}\oplus\left( \frac{a}{b}\oplus\left(\cdots \oplus \frac{a}{b} \oplus\left(\frac{a}{b} \oplus \frac{c}{d}\right) \cdots \right)\right),
\end{eqnarray*}
are in $\Fartwo^{(n+m)}$. Similarly, if $u<v$ are Markov irrationalities that are consecutive in $\M^{(n)}$, then we can form $v^{\odot m}\odot u,v\odot u^{\odot m}\in \M^{(n+m)}$ by iterative conjunction. 

The following lemma will be used in the proof of Theorem \ref{thm:triangluar_ineq_for_consecutive} in Section \ref{sec: triangle ineq for S(w,t)}.

\begin{lemma}\label{lem: monotonicity Markov tree and opposites}
    Let $u<v$ be real quadratic irrationals with purely periodic continued fraction expansion. Then, for all integer $k\geq 1$ we have
    \[u<v\odot u^{\odot k} , \quad  v^{\odot k}\odot u<v, \quad u<u^{\odot k}\odot v, \quad u\odot v^{\odot k}<v.\]
    Moreover, if we further assume that $u,v$ are Markov irrationalities that are neighbors, then $u^\op<v^\op$ and the above inequalities also hold when replacing $u$ and $v$ by $u^\op$ and $v^\op$, respectively.
\end{lemma}
\begin{proof}
    For the first inequalities it is enough to consider the case $k=1$ since the general result follows from iterative application of this case. Write $u=[\overline{a_1;a_2,\ldots,a_r}]$ and $v=[\overline{b_1;b_2,\ldots,b_s}]$ with $r,s\geq 2$ even and minimal. Define
    \(A:=T^{a_1}V^{a_2}\cdots V^{a_r}, B:=T^{b_1}V^{b_2}\cdots V^{b_r}.\)
    Then $A$, $B$ and $BA$ as elements of $\SLN$ act on $\R^+$ with unique fixed points $u,v,v\odot u$, respectively, and these are all attracting fixed points. In particular, $u<A(v)<v$ and $u<B(u)<v$.  Moreover, $BA$  preserves order in $\R^+$, since any matrix in $\SLN$ does.     We conclude
    \(u<B(u)=BA (u)<BA(v)<B(v)=v.\)
    This implies that $v\odot u$, being the attractive fixed point of $BA$, lies in $(u,v)$. This shows that $u<v\odot u<v$. Similarly, using that $A$ and $B$ also preserve  order in $\R^+$ we have
    \(u=A(u)<AB(u)<AB(v)=A(v)<v.\)
    This implies that $u\odot v$ lies in $(u,v)$ because it is the attracting fixed point of $AB$ in $\R^+$, and shows that $u<u\odot v<v$.

Now, assume that $u,v$ are Markov irrationalities that are neighbors. We want to prove that $u^{\op}<v^{\op}$. This follows by induction on the smallest level $n\geq 0$ containing both $u$ and $v$. Indeed, for $n=0$ this is obvious since $u=[\overline{1;1}]=u^\op$ and $v=[\overline{2;2}]=v^\op$ in this case. Assume that $n\geq 1$. Then, either $u$ is a direct descendant of $v$ or $v$ is a direct descendant of $u$. In the first case we write $u=v\odot w$ with $w$ a left neighbour of $v$ in level $n-1$. By inductive hypothesis we have $w^\op<v^\op$. Then, by the first part of the lemma with $k=1$ we get $w^\op \odot v^\op<v^\op$. Since $u^\op=w^\op \odot v^\op$, we conclude $u^\op<v^\op$ as desired. The proof in the second case is analogous, so we omit the details. Finally, the last statement follows directly from the first inequalities applied to $u^{\op}<v^{\op}$. This completes the proof of the lemma.
\end{proof}

\begin{remark}
    Note that for general purely periodic quadratic irrationalities $u,v$ with $u<v$, it is not true that $u^\op<v^\op.$
    For example take $u=[\overline{1;2}]$ and $v=[\overline{1;1}]$.
\end{remark}
The next proposition gives a criteria for the convexity of $\tilde\Lambda_f$  on any three consecutive Farey fractions $\frac{p_1}{q_1}<\frac{p_2}{q_2}<\frac{p_3}{q_3}$ in a given level $\Fartwo^{(n)}$. It also gives a simple description of the Farey fractions $\frac{p_1}{q_1}$ and $\frac{p_3}{q_3}$ in terms of $\frac{p_2}{q_2}$ and its direct parents.

\begin{prop}\label{prop: convexity reduction}
    Let $\frac{p_1}{q_1}<\frac{p_2}{q_2}<\frac{p_3}{q_3}$ be three consecutive Farey fractions in $\Fartwo^{(n)}$ and let $w_1<w_2<w_3$ be the corresponding Markov irrationalities. Let $k\geq 0$ be the unique integer such that  $\frac{p_2}{q_2}$ has minimal level $n-k$. Then, the following properties hold:
\begin{enumerate}
    \item[$(i)$] If we denote by $\frac{P_1}{Q_1}$ and $\frac{P_3}{Q_3}$ be the left and right neighbors of $\frac{p_2}{q_2}$ in $\Fartwo^{(n-k)}$, respectively, then
    \[\frac{p_2}{q_2}=\frac{P_1}{Q_1}\oplus \frac{P_3}{Q_3}, \quad \frac{p_1}{q_1}=\frac{P_1}{Q_1} \oplus \left( \frac{p_2}{q_2} \right)^{\oplus k}, \quad \frac{p_3}{q_3}=\left( \frac{p_2}{q_2} \right)^{\oplus k} \oplus  \frac{P_3}{Q_3}.\]
    \item[$(ii)$] We have $p_1+p_3=(2k+1)p_2$, $q_1+q_3=(2k+1)q_2$ and in particular 
    \[\frac{p_2}{q_2}=\frac{p_1+p_3}{q_1+q_3}.\]
    \item[$(iii)$] Let $W_1$ and $W_2$ be the Markov irrationalities corresponding to $\frac{P_1}{Q_1}$ and $\frac{P_3}{Q_3}$. Then
    \[w_2=W_3\odot W_1, \quad w_1=w_2^{\odot k}\odot W_1, \quad w_3=W_3\odot w_2^{\odot k}.\]
\end{enumerate}
Moreover, the following two statements are equivalent:
\begin{enumerate}
    \item[$(a)$] $\tilde\Lambda_f$ is convex in $\left\{\frac{p_1}{q_1},\frac{p_2}{q_2},\frac{p_3}{q_3}\right\}$.
    \item[$(b)$] We have \[\Lambda_f(w_2)\leq \frac{s(w_1)}{(2k+1)s(w_2)}\Lambda_f(w_1)+\frac{s(w_3)}{(2k+1)s(w_2)}\Lambda_f(w_3).\]
\end{enumerate}
\end{prop}
\begin{proof}
First, note that $\frac{P_1}{Q_1}$ and $\frac{P_3}{Q_3}$ must have minimal level $n-k-1$, hence they are the two direct parents of $\frac{p_2}{q_2}$, i.e. $\frac{p_2}{q_2}=\frac{P_1}{Q_1}\oplus \frac{P_3}{Q_3}$. Now, for each $0\leq i\leq k$, let $x_i$ and $y_i$ be the immediate left and right neighbors of $\frac{p_2}{q_2}$ at level $n-i$, so that in particular $x_0=\frac{p_1}{q_1},y_0=\frac{p_3}{q_3}$ and $x_k=\frac{P_1}{Q_1},y_k=\frac{P_3}{Q_3}$. Now, by construction of the Farey tree, we have
\[x_i =x_{i+1}\oplus \frac{p_2}{q_2},\quad y_i =\frac{p_2}{q_2}\oplus y_{i+1} \quad \text{for all }i=0,\ldots,k-1.\]
This implies item $(i)$. We also conclude
\(p_1 =P_1+kp_2, q_1 =Q_1+kq_2,  p_3 =P_3+kp_2,  q_3 =Q_3+kq_2.\)
Since $p_2=P_1+P_3$, we get
\(p_1+p_3=(P_1+kp_2)+(P_3+kp_2)=(2k+1)p_2.\)
Similarly, $q_1+q_3=(2k+1)q_2$. This proves item $(ii)$. Item $(iii)$ is a direct consequence of $(i)$, the construction of the Markov tree and the parametrization of Markov irrationalities by Farey fractions. Finally, note that
\(\frac{p_2}{q_2}= \left(\frac{q_1}{q_1+q_3}\right)\frac{p_1}{q_1}+\left(\frac{q_3}{q_1+q_3}\right)\frac{p_3}{q_3},\)
hence convexity of $\tilde{\Lambda}_f$ on $\left\{\frac{p_1}{q_1},\frac{p_2}{q_2},\frac{p_3}{q_3}\right\}$ means
\[\tilde{\Lambda}_f\left(\frac{p_2}{q_2}\right)\leq \frac{q_1}{q_1+q_3}\tilde{\Lambda}_f\left(\frac{p_1}{q_1}\right)+\frac{q_3}{q_1+q_3}\tilde{\Lambda}_f\left(\frac{p_3}{q_3}\right).\]
Then, the fact that $(a)$ and $(b)$ are equivalent follows from the formulas $s(w_i)=2q_i$ for $i=1,2,3$, item $(ii)$ and the fact that $w_i$ corresponds to $\frac{p_i}{q_i}$ for $i=1,2,3$. This completes the proof of the proposition.

\end{proof}

\begin{coro}\label{cor:equivalent_formulations_convexity_for_Lambda_f}
    The following two statements are equivalent:
    \begin{enumerate}
        \item[$(i)$] The function $\tilde \Lambda_f:[0,1/2]\cap \mathbb{Q}\to \R$ is convex.
        \item[$(ii)$] For all Markov irrationalities $W_1<w_2<W_3$ with $W_1,W_3$ neighbors and $w_2=W_3\odot W_1$,  and for all integers $k\geq 0$, we have
 \[\Lambda_f(w_2)\leq \frac{s(w_2^{\odot k}\odot W_1)}{(2k+1)s(w_2)}\Lambda_f(w_2^{\odot k}\odot W_1)+\frac{s(W_3\odot w_2^{\odot k})}{(2k+1)s(w_2)}\Lambda_f(W_3\odot w_2^{\odot k}).\]       
    \end{enumerate}
\end{coro}
\begin{proof}
Clearly, convexity of $\tilde{\Lambda}_f$ on $[0,1/2]\cap \mathbb{Q}$ is equivalent to convexity of $\tilde{\Lambda}_f$ on each  level $\Fartwo^{(n)}$ with  $n\geq 1$ (the case $n=0$ being trivial). Now, according to Lemma \ref{lem:3-convexity_implies_convexity} in the appendix, given $n\geq 1$ convexity on $\Fartwo^{(n)}$ is equivalent to convexity on $\left\{\frac{p_1}{q_1},\frac{p_2}{q_2},\frac{p_3}{q_3}\right\}$ for all possible consecutive fractions $\frac{p_1}{q_1}<\frac{p_2}{q_2}<\frac{p_3}{q_3}$ in that level. Then, the result follows from item $(iii)$ and the equivalence between items $(a)$ and $(b)$ in Proposition \ref{prop: convexity reduction}. This proves the corollary.
\end{proof}

In the next section we prove a triangle inequality for the sums $S(w_2,t)$ and $S(w_2^\op,t)$ which in return will imply statement $(ii)$ in Corollary \ref{cor:equivalent_formulations_convexity_for_Lambda_f}  for $\Lambda_f$.

\subsection{A triangle inequality for $S(w,t)$}\label{sec: triangle ineq for S(w,t)}

In this section we prove a triangle inequality property for the sums $S(w,t)$ for $w\in \M$. This is the main step in our proof of the convexity of $\tilde \Lambda_f:[0,1/2]\cap \mathbb{Q}\to \R$ via Corollary  \ref{cor:equivalent_formulations_convexity_for_Lambda_f}. 

\begin{theorem}\label{thm:triangluar_ineq_for_consecutive}
Let $W_1<w_2<W_3$ be Markov irrationalities with $W_1,W_3$ neighbors and $w_2=W_3\odot W_1$. Let $k\geq 0$ be an integer and define $w_1=w_2^{\odot k}\odot W_1$ and $w_3=W_3\odot w_2^{\odot k}$. Then, for all $t\in [\pi/3,2\pi/3]$ we have
    \[(2k+1)S(w_2,t)< S(w_1,t)+S(w_3,t) \text{ and }(2k+1)S(w^\op,t)< S(w_1^\op,t)+S(w_3^\op,t).\]
\end{theorem}

For the proof of Theorem \ref{thm:triangluar_ineq_for_consecutive} we need to introduce additional notation and prove a few lemmas. 
First, given a  real quadratic irrationality $w$ with purely periodic continued fraction expansion $w=[\overline{a_1,\ldots,a_\ell}]$ with $\ell\geq 2$ even and minimal, we put
\(\ell(w):=\ell.\)

Now, given  $x=[\overline{a_1;\ldots,a_r}]$ and~$y=[\overline{b_1;\ldots,b_s}]$ with $r,s\geq 2$ even and such that $x\neq y$ we can repeat their periods (if necessary) and assume $r=s$. Then, we define their \emph{back matching} $\bmatch(x,y)$ and \emph{front matching} $\fmatch(x,y)$ as the number of equal partial quotients $a_i=b_i$ counted from back to front and from front to back, respectively. More precisely, we have
\[\bmatch(x,y)=k \Leftrightarrow a_r=b_r,\ldots,a_{r-k+1}=b_{r-k+1},a_{r-k}\neq b_{r-k},\]
and
\[\fmatch(x,y)=k \Leftrightarrow a_1=b_1,\ldots,a_{k}=b_{k},a_{k+1}\neq b_{k+1},\]
where it is understood that $\bmatch(x,y)=0$ if $a_r\neq b_r$, and $\fmatch(x,y)=0$ if $a_0\neq b_0$. For instance, we have
\[\bmatch([\overline{1;1}],[\overline{2;2,1,1,1,1}])=4, \quad \fmatch([\overline{1;1}],[\overline{2;2,1,1,1,1}])=0.\]

Recall that, for a purely periodic continued fraction $[\overline{a_1;a_2,\ldots,a_\ell}]$ with $a_i\in \{1,2\}$ for all $i$, we have
\[[\overline{a_{i};a_{i+1},\ldots,a_\ell,a_1,\ldots,a_{i-1}}]=\begin{cases}
    \Phi([\overline{a_{i+1};a_{i+2},\ldots,a_\ell,a_1,\ldots,a_{i}}]) & \text{if }a_i=1,\\
    \Psi([\overline{a_{i+1};a_{i+2},\ldots,a_\ell,a_1,\ldots,a_{i}}]) & \text{if }a_i=2.\\
\end{cases}\]
In particular, one can obtain all the cyclic permutations of $[\overline{a_1;a_2,\ldots,a_\ell}]$ by iterative application of $\Phi$ and $\Psi$, where at each step one uses $\Phi$ or $\Psi$ according to the value of the last partial quotient in the period. Hence, for two different purely periodic continued fractions $x$ and $y$ their back matching $\bmatch(x,y)$ gives us the maximal length of a sequence of $\Phi$ and $\Psi$ that can be applied simultaneously to $x$ and $y$,  to obtain cyclic permutations $x'$ and $y'$  of $x$ and $y$, respectively, with $\bmatch(x',y')=0.$ We will use this idea to pairs of Markov irrationalities $w_1,w_2\in \M$ and their opposites $w_1^{\op},w_2^{\op}$. Since $\bmatch(w_1^{\op},w_2^{\op})=\fmatch(w_1,w_2)$, we start by computing the back and front matching between pairs of Markov irrationalities. The next lemma 
gives useful formulas in the case of Markov irrationalities that are neighbors (i.e., consecutive in some level $\M^{(n)}$).

\begin{lemma}\label{lem:back_matching_formula}
Let $u<v$ be two Markov irrationalities that are neighbors. Then, we have
 \(\bmatch(u,v)=\ell(v)-2\) 
 and 
 \(\fmatch(u,v)=\ell(u)-2\).
\end{lemma}
\begin{proof}
In order to prove the first formula we will in fact  prove the formulas
\begin{equation}\label{eq:bmatch(u,v)_ind}
 \bmatch(u,v)=\ell(v)-2=\bmatch(v\odot u,v) 
\end{equation}
by induction on the smallest level $n\geq 0$ of the Markov tree  containing $u$ and $v$, starting with $n=0$ where $u=[\overline{1;1}]$ and $v=[\overline{2;2}]$ (see Figure \ref{fig:Markov_tree}). In this case we have 
\(\bmatch(u,v)=0\) and \(  \bmatch(v\odot u,v)=0, \)
which verifies the desired formulas since $\ell(v)=2$. Now, assume that the smallest level of the Markov tree containing $u$ and $v$ is $n\geq 1$. Since $u$ and $v$ are neighbour and $n\geq 1$, we have either $u=v\odot w$ with $w$ a left neighbour of $v$ in level $n-1$, or $v=w\odot u$ with $w$ a right  neighbour of $u$ in level $n-1$. Let us first assume $u=v\odot w$. 
\[\begin{tikzcd}
  w  \arrow[rr,  dotted, no head, "neighbors" description]  & & v \arrow[ld, no head] \\
   & u=v\odot w & 
 \end{tikzcd}\]
Then
\(
    \bmatch(u,v)  =  \bmatch(v\odot w,v).
\)
By induction, we know that $\bmatch(w,v)=\bmatch(v\odot w,v)=\ell(v)-2$. This already proves that~$\bmatch(u,v)=\ell(v)-2$. We also have
\(
    \bmatch(v\odot u,v)  =  \bmatch(v\odot v\odot w,v).
\)
We know that $\bmatch(v\odot w,v)=\ell(v)-2<\ell(v)$, and this implies that when comparing the partial quotients of $v\odot w$ and $v$ from back to front, we don't  go beyond $\ell(v)$ steps from back to front. Since $\ell(v\odot w)>\ell(v)$, we conclude $\bmatch(v\odot v\odot w,v)=\bmatch(v\odot w,v)$. This implies
\(
    \bmatch(v\odot u,v)   =  \bmatch(v\odot v\odot w,v)= \bmatch(v\odot w,v)=\ell(v)-2.
\)
This proves \eqref{eq:bmatch(u,v)_ind} in the case where $v$ is a parent of $u$. 

We now treat the case $v=w\odot u$. 
\[\begin{tikzcd}
u \arrow[rd, no head] \arrow[rr, dotted, no head, "neighbors" description] & & w  \\
   & v=w\odot u & 
 \end{tikzcd}\]
We first compute \(
    \bmatch(u,v)  =  \bmatch(u,w\odot u) = \ell(u)+\bmatch(u,u\odot w).
\)
Since $u,w$ appear in a previous level, we have by induction $\bmatch(u,w)=\ell(w)-2<\ell(w)$. Hence $\bmatch(u,u\odot w)=\bmatch(u, w)$ and we have 
\(
    \bmatch(u,v)  = \ell(u)+\bmatch(u, w)= \ell(u)+\ell(w)-2=\ell(v)-2.
\)
Similarly, we compute
\(
    \bmatch(v\odot u,v)  = \bmatch(w\odot u\odot u,w\odot u)=\ell(u)+\bmatch(u\odot w\odot u,u\odot w). 
\)
By induction we also know that $\bmatch(w\odot u,w)=\ell(w)-2<\ell(w)$, which implies
\(\bmatch(v\odot u,v)=\ell(u)+\ell(w)-2=\ell(v)-2\)
as desired. This completes the proof of \eqref{eq:bmatch(u,v)_ind} and of the first stated formula. The corresponding result for the front matching is obtained in an analogous way, so we omit the details. 
\end{proof}

Given a quadratic irrationality $w=[\overline{a_1;a_2,\ldots,a_\ell}]$ with $\ell\geq 2$ even and minimal, we extend the definition of $a_i$ to any $i\in \Z$ by putting $a_{i}:=a_{i_0}$ if $i\equiv i_0$ mod $\ell$ with $1\leq i_0\leq \ell$. Moreover, for $i\in \Z$ we put
\(w_{(i)}:=[\overline{a_{i+1};\ldots,a_{\ell},a_1,\ldots,a_{i-1},a_i}].\)
For example $w_{(\ell)}= w$, $w_{(\ell-1)}= [\overline{a_\ell;a_1,\ldots,a_{\ell-1}}], \ldots, w_{(1)}=  [\overline{a_2;a_3,\ldots,a_\ell,a_1}]$, $w_{(0)}= w$.

\begin{lemma}\label{lem: comparison of cyclic permutations}
Let $u<v$ be real quadratic irrationals with purely periodic continued fraction expansion and $0\leq i \leq \bmatch (u,v)$ an integer. Then, $u_{(\ell(u)-i)}<v_{(\ell(v)-i)}$ if $i$ is even, and $u_{(\ell(u)-i)}>v_{(\ell(v)-i)}$ if $i$ is odd. 
\end{lemma}
\begin{proof}
We can assume $i\geq 1$. For every integer $a\geq 1$ the matrix $\eta^{(a)}:=\left(\begin{smallmatrix}
    a & 1 \\ 1 & 0
\end{smallmatrix}\right)$ acts on $\R^+$ as $x\mapsto a+\frac{1}{x}$ and the map $\eta^{(a)}:\R^+\to \R^+$ is decreasing. Moreover, for a purely periodic continued fraction $[\overline{a_1;a_2,\ldots,a_\ell}]$ we have
\(\eta^{(a_\ell)}([\overline{a_1;a_2,\ldots,a_\ell}])=[\overline{a_\ell;a_1,\ldots,a_{\ell-1}}].\)
Then, writing $u=[\overline{c_1;c_2,\ldots,c_r,a_1,\ldots,a_N}]$, $v=[\overline{d_1;d_2,\ldots,d_s,a_1,\ldots,a_N}]$ with $N:=\bmatch (u,v)$,  the result follows from the formulas
\(u_{(\ell(u)-i)}= \eta^{(a_{N-i+1)}}\circ \cdots \circ \eta^{(a_N)}(u)\), \( v_{(\ell(v)-i)}= \eta^{(a_{N-i+1})}\circ \cdots \circ \eta^{(a_N)}(v).\) 
This proves the lemma.
\end{proof}

Recall that the sum $S(w,t)$ for $t\in [\pi/3,2\pi/3]$ is defined in \eqref{eq:def_of_S(w,t)} and can be rewritten  as
\(S(w,t)=\sum\limits_{i=1}^{\ell(w)}S_i(w,t)\)
with $S_i(w,t)$ defined in \eqref{eq:def_of_S_i(w,t)}. Moreover, we extended the definition of $S_i(w,t)$ to $i\in \Z$ by putting $S_i(w,t):=S_{i_0}(w,t)$ where $i_0$ is the unique integer satisfying $1\leq i_0\leq \ell$ and $i\equiv i_0$ mod $\ell$.

The idea of the proof of Theorem \ref{thm:triangluar_ineq_for_consecutive} for the first inequality \[(2k+1)S(w_2,t)<S(u,t)+S(v,t),\]
is to find a convenient splitting of $(2k+1)S(w_2,t)$ of the form
\[(2k+1)S(w_2,t)=\sum_{i=n_1}^{n_2}S_i(w_2,t)+\sum_{i=m_1}^{m_2}S_i(w_2,t),\]
for some integers $n_1\leq n_2$ and $m_1\leq m_2$, such that
\begin{equation}\label{eq:convenient splitting}
\sum_{i=n_1}^{n_2}S_i(w_2,t)<S(u,t) \quad \text{and} \quad \sum_{i=m_1}^{m_2}S_i(w_2,t)<S(v,t) \quad \text{for all }t\in [\pi/3,2\pi/3].
\end{equation}
Analogously, for the second inequality we will find a convenient splitting of $(2k+1)S(w_2^{\op},t)$.

In our proof of Theorem \ref{thm:triangluar_ineq_for_consecutive} the following  lemma will be used to obtain comparisons of the form \eqref{eq:convenient splitting}. 

\begin{lemma}\label{lem:triang_ineq_on_pairs}
   Let $x=[\overline{c_1;c_2,\ldots,c_r}]$, $y=[\overline{d_1;d_2,\ldots,d_s}]$ with $ x<y$, $2\leq r,s$ with $r,s$ even and minimal, and $c_i,d_j\in \{1,2\}$ for all $i,j$. Then, the following properties hold for all $t\in [\pi/3,2\pi/2]$:
\begin{enumerate}
    \item[$(i)$] If $\bmatch(x,y)=r-1$, $c_i=c_{i+1}$ and $d_i=d_{i+1}$ for all $i$ even, and  $c_1=2$ (equivalently $d_{s-r+1}=1$), then
\(\sum\limits_{i=s-r}^{s}S_{i}(y,t)< S(x,t).\)
\item[$(ii)$] If $\bmatch(x,y)=r-2$, $c_i=c_{i+1}$  and $d_i=d_{i+1}$ for all $i$ odd, and $c_1=1$ (equivalently $d_{s-r+1}=2$), then
\(\sum\limits_{i=s-(r-2)}^{s}S_{i}(y,t)< S(x,t).\)
\end{enumerate}
\end{lemma}
\begin{proof}
The proof of this lemma is based on an iterative application of Lemmas \ref{lem:key_ineq_FT} and \ref{lem:key_ineq_FT2}. For $(i)$ it is enough to prove that
\begin{equation}\label{eq:S_{i}(y,t)+S_{i-1}(y,t)_i_even}
   S_{s-i}(y,t)+S_{s-i-1}(y,t) < S_{r-i}(x,t)+S_{r-i-1}(x,t) \quad \text{for all }0\leq i\leq r-4\text{ even},   
\end{equation}
and
\begin{equation}\label{eq:S_{s-r}(y,t)+S_{s-r+1}(y,t)+S_{s-r+2}(y,t)}
   S_{s-r}(y,t)+S_{s-r+1}(y,t)+S_{s-r+2}(y,t) < S_1(x,t)+S_2(x,t).     
\end{equation}
In order to prove \eqref{eq:S_{i}(y,t)+S_{i-1}(y,t)_i_even}, for $z\in\{x,y\}$ and $\ell=\ell(z)\in \{r,s\}$ we write $z=[\overline{a_1;a_2,\ldots,a_\ell}]$.  Let $i$ be an even integer with $0\leq i\leq r-4$ and note that $4/3<x_{(r-i)}<y_{(s-i)}$ (by Remark \ref{rmk:4/3} and Lemma \ref{lem: comparison of cyclic permutations}) and $c_{r-i}=d_{s-i}$, $c_{r-i-1}=d_{s-i-1}$. For the computation of $S_{\ell-i}(z,t)+S_{\ell-i-1}(z,t)$ for $z\in \{x,y\}$ we consider four cases according to the values of $c_{r-i}$ and $c_{r-i-1}$:
\begin{itemize}[leftmargin=*]
    \item If $c_{r-i}=c_{r-i-1}=1$, then
\[     S_{\ell-i}(z,t)+S_{\ell-i-1}(z,t)= F(z_{(\ell-i)},t)+F(\Phi(z_{(\ell-i)}),t).\]
\item If $c_{r-i}=c_{r-i-1}=2$, then
\[S_{\ell-i}(z,t)+S_{\ell-i-1}(z,t)= F(z_{(\ell-i)},t)+F(\Phi(z_{(\ell-i)}),t)+F(\Psi(z_{(\ell-i)}),t)+F(\Phi\circ \Psi(z_{(\ell-i)}),t).\]
Moreover, since also $c_{r-i+1}=2$ we have  $2<x_{(r-i)}<y_{(s-i)}$.
    \item If $c_{r-i}=1$ and $c_{r-i-1}=2$, then
\[S_{\ell-i}(z,t)+S_{\ell-i-1}(z,t)=F(z_{(\ell-i)},t)+F(\Phi(z_{(\ell-i)}),t)+F(\Phi\circ \Phi(z_{(\ell-i)}),t).\]
    \item If $c_{r-i}=2$ and $c_{r-i-1}=1$, then  
\[S_{\ell-i}(z,t)+S_{\ell-i-1}(z,t)=F(z_{(\ell-i)},t)+F(\Phi(z_{(\ell-i)}),t)+F(\Psi(z_{(\ell-i)}),t).\]
Moreover, we also have $c_{r-i+1}=2$ and $2< x_{(r-i)}<y_{(s-i)}$, as in the second case above.
\end{itemize}
In each case the inequality \eqref{eq:S_{i}(y,t)+S_{i-1}(y,t)_i_even} follows from the facts that $u\mapsto F(u,t)+F(\Phi(u),t)$ is decreasing for $u\in [4/3,\infty)$, $u\mapsto F(u,t)+F(\Phi(u),t)+F(\Psi(u),t)$ is decreasing for $u\in [\phi,\infty)$ and $u\mapsto F(u,t)$ is decreasing for $u\in [1,\infty)$ (Lemmas \ref{lem:key_ineq_FT}, \ref{lem:key_ineq_FT2} and \ref{lem:prop_of_F}, respectively). This proves \eqref{eq:S_{i}(y,t)+S_{i-1}(y,t)_i_even}.

We now turn to the proof of \eqref{eq:S_{s-r}(y,t)+S_{s-r+1}(y,t)+S_{s-r+2}(y,t)}. Recall that $c_1=2$ and note that $x_{(2)}<y_{(s-r+2)}$ (by Lemma \ref{lem: comparison of cyclic permutations}). We consider two cases depending on the value of $c_2=c_3$:
\begin{itemize}[leftmargin=*]
    \item If $c_2=1$, then
    \(x_{(2)}=[\overline{1;c_4,\ldots,c_{r-1},2,2,1}]<y_{(s-r+2)}=[\overline{1;d_{s-r+4},\ldots,d_{s-r-1},1,1,1}]\)
    and
\begin{eqnarray*}
    S_{s-r}(y,t)+S_{s-r+1}(y,t)+S_{s-r+2}(y,t)&=& F(y_{(s-r+2)},t)+F(\Phi(y_{(s-r+2)}),t)\\
    & & +F(\Phi\circ \Phi(y_{(s-r+2)}),t), \\
    S_1(x,t)+S_2(x,t) &=& F(x_{(2)},t)+F(\Phi(x_{(2)}),t) +F(\Phi\circ \Phi(x_{(2)}),t).
\end{eqnarray*}
\item If $c_2=2$, then
\(x_{(2)}=[\overline{2;c_4,\ldots,c_{r-1},2,2,2}]<y_{(s-r+2)}=[\overline{2;d_{s-r+4},\ldots,d_{s-r-1},1,1,2}]\)
and
\begin{eqnarray*}
    S_{s-r}(y,t)+S_{s-r+1}(y,t)+S_{s-r+2}(y,t)&=& F(y_{(s-r+2)},t)+F(\Phi(y_{(s-r+2)}),t)\\
    & & +F(\Psi(y_{(s-r+2)}),t)+F(\Phi\circ \Psi(y_{(s-r+2)}),t), \\
    S_1(x,t)+S_2(x,t) &=& F(x_{(2)},t)+F(\Phi(x_{(2)}),t)\\
    & & +F(\Psi(x_{(2)}),t)+F(\Phi\circ \Psi(x_{(2)}),t).
\end{eqnarray*}
\end{itemize}
In each case \eqref{eq:S_{s-r}(y,t)+S_{s-r+1}(y,t)+S_{s-r+2}(y,t)} follows from Lemmas \ref{lem:key_ineq_FT}, \ref{lem:key_ineq_FT2} and \ref{lem:prop_of_F}, as argued in the proof of \eqref{eq:S_{i}(y,t)+S_{i-1}(y,t)_i_even}.

The proof of $(ii)$ is similar. It reduces to the proof of \eqref{eq:S_{i}(y,t)+S_{i-1}(y,t)_i_even} and of
\begin{equation}\label{eq:S_{s-r+2}(y,t)}
   S_{s-r+2}(y,t) < S_1(x,t)+S_2(x,y).     
\end{equation}
For \eqref{eq:S_{i}(y,t)+S_{i-1}(y,t)_i_even} we can use the same proof given above for $(i)$ noting that $\phi \leq x_{(i)}<y_{(i)}$ for all $i$ even (by Lemma \ref{lem: comparison of cyclic permutations}), although now only the cases $c_{r-i}=c_{r-i-1}=1$ and $c_{r-i}=c_{r-i-1}=2$ can occur. For \eqref{eq:S_{s-r+2}(y,t)} we note that
\(\frac{4}{3}<x_{(2)}=[\overline{c_3;\ldots,c_r,1,1}]<y_{(s-r+2)}=[\overline{d_{s-r+3};\ldots,d_{s-r},2,2}]\)
and
\begin{eqnarray*}
    S_{s-r+2}(y,t)&=& F(y_{(s-r+2)},t)+F(\Phi(y_{(s-r+2)}),t),\\
    S_1(x,t)+S_2(x,t) &=& F(x_{(2)},t)+F(\Phi(x_{(2)}),t).
\end{eqnarray*}
Hence, \eqref{eq:S_{i}(y,t)+S_{i-1}(y,t)_i_even} follows from Lemma \ref{lem:key_ineq_FT}. This completes the proof of the lemma.
\end{proof}

Note that for every quadratic irrationality $w$ with purely periodic continued fraction expansion we have
\begin{equation}\label{eq: sum S_i(w_(l+k),t)}
    \sum_{i=n}^mS_i(w_{(k)},t)=\sum_{i=n+k}^{m+k}S_i(w,t) \quad \text{for all }n,m,k\in \Z \text{ with }n\leq m,
\end{equation}
and
\begin{equation}\label{eq: sum S(w_i,t)}
    S(w_{(i)},t)=S(w,t) \quad \text{for all }i\in \Z.
\end{equation}

\medskip

We now present our proof of Theorem \ref{thm:triangluar_ineq_for_consecutive}.

\begin{proof}[Proof of Theorem \ref{thm:triangluar_ineq_for_consecutive}]
Fix $t\in [\pi/3,2\pi/3]$. We start by sketching the idea of the proof in a simple case. Namely, we assume first that $k=0$, so that $W_1=w_1=[\overline{a_1;\ldots,a_n}]$ and $W_3=w_3=[\overline{b_1;\ldots,b_n}]$ are neighbors and $w_2=w_3\odot w_1=[\overline{b_1;\ldots,b_m,a_1,\ldots,a_n}]$. Moreover, we assume $n>m$ for simplicity. As usual, we have
\(w_1<w_2<w_3\)
and we want to prove that
\begin{equation}\label{eq: simple case first triang ineq}
    S(w_2,t)<S(w_1,t)+S(w_3,t).
\end{equation}
Since $x\mapsto F(x,t)$ is decreasing for $x\geq 1$ the terms $F(w_2,t)$ and $F(w_3,t)$ in $S(w_2,t)$ and $S(w_3,t)$, respectively, satisfy $F(w_2,t)>F(w_3,t)$. Clearly, this does not help in proving \eqref{eq: simple case first triang ineq}. To go around this problem we note that $b_1=2$ and use the map $\Psi^{-1}$, which reverses order in $\R^+$, to get
\[\Psi^{-1}(w_2)=[\overline{b_2;\ldots,b_m,a_1,\ldots,a_n,b_1}]>\Psi^{-1}(w_3)=[\overline{b_2;\ldots,b_m,b_1}].\]
Calling $x:=[\overline{b_2;\ldots,b_m,b_1}]$ and $y:=[\overline{b_2;\ldots,b_m,a_1,\ldots,a_n,b_1}]$, we compute
\(\bmatch(x,y)=1+\bmatch(w_2,w_3)=1+(m-2)=m-1.\) 
This shows that Lemma \ref{lem:triang_ineq_on_pairs}$(i)$ is applicable, giving the inequality
\[\sum_{i=n}^{n+m}S_i(\Psi^{-1}(w_2),t)<S(\Psi^{-1}(w_3),t),\]
or equivalently
\begin{equation}\label{eq: final first ineq easy case}
\sum_{i=n+1}^{n+m+1}S_i(w_2,t)<S(w_3,t).    
\end{equation}
Now, we look at $w_1<w_2$. Although this is the right direction of inequality, in order to obtain an inequality that complements \eqref{eq: final first ineq easy case} we find it necessary to apply a cyclic permutations to $w_1$ and $w_2$. More precisely, we consider
\[(w_1)_{(n-m)}=[\overline{a_{n-m+1},\ldots,a_n,a_1,\ldots,a_{n-m}}]< (w_2)_{(n)}=[\overline{a_{n-m+1},\ldots,a_n,b_1,\ldots,b_m,a_1,\ldots,a_{n-m}}].\]
Since $\bmatch(w_1,w_2)=n+m-2$, we have
\(\bmatch((w_1)_{(n-m)},(w_2)_{(n)})=(n+m-2)-m=n-2.\) 
So, Lemma \ref{lem:triang_ineq_on_pairs}$(ii)$ is applicable with $x=(w_1)_{(n-m)},y=(w_2)_{(n)}$ and we obtain
\[\sum_{i=m+2}^{n+m}S_i((w_2)_{(n)},t)<S((w_1)_{(n-m)},t),\]
or equivalently
\[\sum_{i=n+m+2}^{2n+m}S_i(w_2,t)<S(w_1,t).\]
Finally, as we will see in the general case below, we have 
\[\sum_{i=n+1}^{n+m+1}S_i(w_2,t)+\sum_{i=n+m+2}^{2n+m}S_i(w_2,t)=S(w_2,t).\]

We now move to the general case. Let $W_1<w_2<W_3$ be Markov irrationalities with $W_1,W_3$ neighbors  and $w_2=W_3\odot W_1$. Let $k\geq 0$, $w_1:=w_2^{\odot k}\odot W_1$, $w_3:=W_3\odot w_2^{\odot k}$. Put $n:=\ell(W_1)$ and $m:=\ell(W_3)$ so that $\ell(w_2)=n+m$. We first claim that
 \begin{eqnarray}
\sum_{i=n+1-k(n+m)}^{n+m+1}S_{i}(w_2,t)&<&S(w_3,t), \label{eq:first_sum<S(w_3,t)} \\
\sum_{i=(1-k)(n+m)+2}^{2n+m}S_{i}(w_2,t)&<&S(w_1,t), \label{eq:second_sum<S(w_1,t)}
 \end{eqnarray}
 and
 \begin{equation}\label{eq: first splitting proof}
(2k+1)S(w_2,t)=\sum_{i=n+1-k(n+m)}^{n+m+1}S_{i}(w_2,t)+\sum_{i=2-k(n+m)}^{n}S_{i}(w_2,t).
\end{equation}

Clearly these properties imply the first inequality
\begin{equation}\label{eq: first ineq proof convexity}
(2k+1)S(w_2,t)<S(w_1,t)+S(w_3,t).    
\end{equation}

 In order to prove \eqref{eq:first_sum<S(w_3,t)} note that  $w_2<w_3$ and both $w_2,w_3$ have first partial quotient equal to $2$. Hence, $\Psi^{-1}(w_3)<\Psi^{-1}(w_2)$ and $\Psi^{-1}(w_3),\Psi^{-1}(w_2)$ have last partial quotient in their periods equal to $2$. We will apply Lemma \ref{lem:triang_ineq_on_pairs}$(i)$ to $x=\Psi^{-1}(w_3)$, $y=\Psi^{-1}(w_2)$ with $r=k(n+m)+m=\ell(w_3)=\ell(x)$ and $s=n+m=\ell(w_2)=\ell(y)$. In order to do that we compute
 \(\bmatch(\Psi^{-1}(w_3),\Psi^{-1}(w_2))=1+\bmatch(w_2,w_3)=\ell(w_3)-1=\ell(x)-1,\) 
where in the last equality we used Lemma \ref{lem:back_matching_formula} noting that $w_2<w_3$ are neighbors. Thus, the hypotheses in Lemma \ref{lem:triang_ineq_on_pairs}$(i)$ are satisfied and we obtain 
\[\sum_{i=n-k(n+m)}^{n+m}S_{i}(\Psi^{-1}(w_2),t)<S(\Psi^{-1}(w_3),t).\]
Since $\Psi^{-1}(w_2)=(w_2)_{(1)}$ and $\Psi^{-1}(w_3)=(w_3)_{(1)}$, using \eqref{eq: sum S_i(w_(l+k),t)} and \eqref{eq: sum S(w_i,t)} we get
\[\sum\limits_{i=n+1-k(n+m)}^{n+m+1}S_{i}(w_2,t)<S(w_3,t).\]
This proves \eqref{eq:first_sum<S(w_3,t)}. 

Now, define $U:=W_1\odot w_2^k$ and $W:=w_2$. We will apply Lemma \ref{lem:triang_ineq_on_pairs}$(ii)$ to $x=U_{(\ell(U)-m)}$, $y=W_{(n)}$ with $r=k(n+m)+n$ and  $s=n+m$. For this, we first note that the partial quotient of $W_{(n)}$ of index $r-s+1=k(n+m)-m+1$ is just the first partial quotient of $w_2$, hence it is equal to $2$. Now, we compute (using Lemma \ref{lem:back_matching_formula})
\(\bmatch(U,W)=k(n+m)+\bmatch(w_2^{\odot k}\odot w_1,w_2)=k(n+m)+n+m-2.\)
This implies 
\(\bmatch(U_{(\ell(U)-m)},W_{(n)})=k(n+m)+n-2=\ell(U)-2.\)
Moreover, we have $U<W$ (by Lemma \ref{lem: monotonicity Markov tree and opposites}) hence $U_{(\ell(U)-m)}<W_{(n)}$ (by Lemma \ref{lem: comparison of cyclic permutations}). This shows that the hypotheses in Lemma \ref{lem:triang_ineq_on_pairs}$(ii)$ are satisfied and we obtain
\[\sum_{i=m-k(n+m)+2}^{n+m}S_{i}(W_{(n)},t)<S(U_{(\ell(U)-m)},t)=S(U,t).\]
Since $U=(w_1)_{(\ell(w_1)-n)}$, by \eqref{eq: sum S_i(w_(l+k),t)} we get
\[\sum_{i=(1-k)(n+m)+2}^{2n+m}S_{i}(w_2,t)<S(w_1,t).\]
This proves \eqref{eq:second_sum<S(w_1,t)}. 

We now have to prove \eqref{eq: first splitting proof}. Using that $S_i(w_2,t)$ is periodic in $i$ of period $n+m$, we get
\begin{eqnarray*}
& & \sum_{i=n+1-k(n+m)}^{n+m+1}S_{i}(w_2,t)+\sum_{i=(1-k)(n+m)+2}^{2n+m}S_{i}(w_2,t)\\
&=&\sum_{i=n+1}^{(k+1)(n+m)+1}S_{i}(w_2,t)+\sum_{i=2}^{k(n+m)+n}S_{i}(w_2,t)\\
&=&\left(kS(w_2,t)+\sum_{i=n+1}^{n+m+1}S_{i}(w_2,t)\right)+\left(kS(w_2,t)+\sum_{i=2}^{n}S_{i}(w_2,t)\right)\\
&=& (2k+1)S(w_2,t).
\end{eqnarray*}
This proves  \eqref{eq: first splitting proof} and completes the proof of \eqref{eq: first ineq proof convexity}.

The proof of the second inequality
\begin{equation}\label{eq: second ineq proof convexity}
  (2k+1)S(w_2^{\op},t)<S(w_1^{\op},t)+S(w_3^{op},t)  
\end{equation}
is similar and it reduces to the proof of
 \begin{eqnarray}
\sum_{i=(1-k)(n+m)+1}^{n+2m+1}S_i(w_2^{\op},t)&<&S(w_3^{\op},t), \label{eq:first_sum<S(w_3op,t)} \\
\sum_{i=m-k(n+m)+2}^{n+m}S_i(w_2^{\op},t)&<&S(w_1^{\op},t), \label{eq:second_sum<S(w_1op,t)}
 \end{eqnarray}
 and
\begin{equation}\label{eq: second splitting proof}
(2k+1)S(w_2^\op,t)<\sum_{i=1-k(n+m)}^{m+1}S_i(w_2^{\op},t)+\sum_{i=n-k(n+m)+2}^{n+m}S_i(w_2^{\op},t).
\end{equation}

To prove \eqref{eq:first_sum<S(w_3op,t)} we start by observing that  $w_2^{\op}<W_3^{\op}\odot (w_2^{\op})^{\odot k}$ (by Lemma \ref{lem: monotonicity Markov tree and opposites}). In this case we put $Z:=w_2^{\op}$ and $V:=W_3^{\op}\odot (w_2^{\op})^{\odot k}$. We will apply Lemma \ref{lem:triang_ineq_on_pairs}$(i)$ to $x=V_{(\ell(V)-n+1)}$, $y=Z_{(m+1)}$ with $r=k(n+m)+m$ and $s=n+m$. A simple computation shows that the partial quotient of $Z_{(m+1)}$ of index $s-r+1$ equals the second partial quotient of $w_2^{\op}$, which is equal to 1. Now, computing (using Lemma \ref{lem:back_matching_formula})
\(\bmatch(Z,V)=\fmatch(w_2,w_2^{\odot k}\odot W_3)=(k+1)(n+m)-2,\)
we deduce that
\(\bmatch(Z_{(m+1)},V_{(\ell(V)-n+1)})=k(n+m)+m-1=\ell(V_{(\ell(V)-n+1)})-1.\)
Moreover, since $Z<V$ we have $Z_{(m+1)}>V_{(\ell(V)-n+1)}$ (by Lemma \ref{lem: comparison of cyclic permutations}). This shows that all the hypotheses in Lemma \ref{lem:triang_ineq_on_pairs}$(i)$ are satisfied, so we deduce that
\begin{equation*}
\sum_{i=n-k(n+m)}^{n+m}S_i(Z_{(m+1)},t)<S(V_{(\ell(V)-n+1)},t)=S(V,t)=S(w_3^{\op},t),
\end{equation*}
where in the last two equalities we used \eqref{eq: sum S(w_i,t)} and the fact that $V=(w_3^{\op})_{(\ell(w_3)-m)}$. Then, by \eqref{eq: sum S_i(w_(l+k),t)} the left hand side equals $\sum\limits_{i=(1-k)(n+m)+1}^{n+2m+1}S_i(w_2^{\op},t)$ and this shows \eqref{eq:first_sum<S(w_3op,t)}.

For the proof of \eqref{eq:second_sum<S(w_1op,t)} note that $w_1^{\op}<Z$ (by Lemma \ref{lem: monotonicity Markov tree and opposites}). We will use Lemma \ref{lem:triang_ineq_on_pairs}$(ii)$ with $x=w_1^{\op}$, $y=Z$, $r=k(n+m)+n$ and $s=n+m$. For this, observe that the first partial quotient of $w_1^{\op}$ is the last partial quotient in the period of $w_1$, and this is equal to $1$. Next, we compute
\(\bmatch(w_1^{\op},Z)=\fmatch(w_2^{\odot k}\odot W_1,w_2)=k(n+m)+n-2=\ell(w_1^{\op})-2\)
by Lemma \ref{lem:back_matching_formula}. By Lemma \ref{lem:triang_ineq_on_pairs}$(ii)$ we conclude
\begin{equation*}
\sum_{i=m-k(n+m)+2}^{n+m}S_i(Z,t)<S(w_1^{\op},t).
\end{equation*}
The left hand side equals $\sum\limits_{i=m-k(n+m)+2}^{n+m}S_i(w_2^{\op},t)$ and this proves \eqref{eq:second_sum<S(w_1op,t)}. Finally, since
\begin{eqnarray*}
&& \sum_{i=(1-k)(n+m)+1}^{n+2m+1}S_i(w_2^{\op},t)+\sum_{i=m-k(n+m)+2}^{n+m}S_i(w_2^{\op},t)\\
&=& \sum_{i=1}^{k(n+m)+m+1}S_i(w_2^{\op},t)+\sum_{i=m+2}^{k(n+m)+n+m}S_i(w_2^{\op},t)\\
&=& \left(kS(w_2^\op,t)+\sum_{i=1}^{m+1}S_i(w_2^{\op},t)\right)+\left(kS(w_2^\op,t)+\sum_{i=m+2}^{n+m}S_i(w_2^{\op},t)\right)\\
&=& (2k+1)S(w_2^{\op},t),
\end{eqnarray*}
we get \eqref{eq: second splitting proof} and conclude that the inequality \eqref{eq: second ineq proof convexity} holds. This completes the proof of the theorem.
\end{proof}

\subsection{Proof of Theorem \ref{thm:Lambda_f_convexity}}

We will apply Lemma \ref{lem:extension_of_convex} in the appendix to $G=\tilde{\Lambda}_f:[0,1/2]\cap \mathbb{Q}\to \R^+$. By Theorem \ref{thm:triangluar_ineq_for_consecutive}, Proposition \ref{prop:formula_for_Re(I_f(A))} and Corollary \ref{cor:equivalent_formulations_convexity_for_Lambda_f} we have that $\tilde \Lambda_f:[0,1/2]\cap \mathbb{Q}\to \R^+$ is convex. Moreover, $\tilde \Lambda_f(x)$ has a global maximum at $x=0$ and a global minimum at $x=1/2$, by Theorem \ref{thm:Lambda_f(M)}. Finally, defining
\(x_n:=\left(\frac{0}{1}\right)^{\oplus n}\oplus  \frac{1}{2},\)
we get
\[\Lambda_f(x_n)=\frac{\Re(I_f(T^2V^2(TV)^n)}{4+2n}\to \frac{\Re(I_f(TV))}{2}=\Lambda_f(\phi)=\tilde \Lambda_f(0)\]
by Lemma \ref{lem:I_f(AB^m)_goes_to_I_f(B)}. Hence, we can apply Lemma \ref{lem:extension_of_convex} to obtain a unique continuous, convex, non-increasing extension $\tilde \Lambda_f:[0,1/2]\to \R^+$. Finally, by Remark \ref{rmk:silver_ratio_strict_minimum} we have that $\tilde \Lambda_f(x)>\tilde \Lambda_f(1/2)$ for all $x\in [0,1/2)\cap \mathbb{Q}$, and this extends to all $x\in [0,1/2)$ since $\tilde \Lambda_f$ is non-increasing. Thus, by Lemma \ref{lem:conv_and_global_min_implies_decreasing} in the appendix we deduce that $\tilde \Lambda_f$ is actually decreasing. This proves Theorem \ref{thm:Lambda_f_convexity}.

\section*{Appendix}

\subsection*{A.1.~Proof of Lemma \ref{lem:key_ineq_FT}}

Define 
 \(Z(x,t):=F(x,t)+F(\Phi(x),t).\)
 We need to check that $\partial_x Z(x,t)< 0$ for $x\in (4/3,\infty)$. A computation using Lemma \ref{lem:prop_of_F}$(iii)$ shows that $\partial_x Z(x,t)< 0$ is equivalent to
 \begin{eqnarray*}
 \frac{x^2-1}{2x+1} > \left(\frac{1+x^2-2x\cos(t)}{2x^2+2x+1-2x(x+1)\cos(t)}\right)^2.
 \end{eqnarray*}
For fixed $x\in (4/3,\infty)$ the function
\begin{equation}\label{eq:function_of_t}
    y \mapsto \frac{1+x^2-2xy}{2x^2+2x+1-2x(x+1)y} \quad \text{ for }y\in [-1/2,1/2]
\end{equation}
has derivative
\begin{equation}\label{eq:der_wrt_t}
-\frac{2x^2(x^2-x-1)}{(2x^2+2x+1-2x(x+1)y)^2}.
\end{equation}
We distinguish two cases: If $x\in [\phi,\infty)$, then $x^2-x-1\geq 0$ and \eqref{eq:der_wrt_t} is non-positive, which implies that \eqref{eq:function_of_t} is non-increasing and therefore
    \[\frac{1+x^2-2x\cos(t)}{2x^2+2x+1-2x(x+1)\cos(t)}\leq \frac{1+x^2-x}{x^2+x+1}.\]
Hence, in this case, it is enough to prove the inequality
\[\frac{x^2-1}{2x+1} >\left(\frac{1+x^2-x}{x^2+x+1}\right)^2 \quad \text{for }x\geq \phi.\]
A straightforward computation shows that this is equivalent to
\begin{equation*}
    p(x):=x^6 + 5 x^4 - 4 x^3 - x^2 - 2 x - 2> 0 \quad \text{for }x\geq \phi.
\end{equation*}
Since $p(x)= \left(x^5 + \frac{3}{2} x^4 + \frac{29}{4} x^3 + \frac{55}{8} x^2 + \frac{149}{16} x + \frac{383}{32}\right)\left(x - \frac{3}{2}\right) + \frac{1021}{64}$
and $\phi \geq \frac{3}{2}$, we conclude the desired result. If $x\in (4/3,\phi)$, then $x^2-x-1< 0$ and \eqref{eq:der_wrt_t} is positive, which implies that \eqref{eq:function_of_t} is increasing and therefore
    \[\frac{1+x^2-2x\cos(t)}{2x^2+2x+1-2x(x+1)\cos(t)}\leq \frac{1+x^2+x}{3x^2+3x+1}.\]
    Arguing as above, the problem is  reduced to proving 
\[q(x):=9 x^6 + 16 x^5 + x^4 - 20 x^3 - 21 x^2 - 10 x - 2>0 \quad \text{for }x\in(4/3,\phi).\]
But
\(q(x)= \left(9 x^5 + 28 x^4 + \frac{115 x^3}{3} + \frac{280 x^2}{9} + \frac{553 x}{27} + \frac{1402}{81}\right)\left(x - \frac{4}{3}\right) + \frac{5122}{243},\)
hence the result follows. 

\subsection*{A.2.~Proof of Lemma \ref{lem:key_ineq_FT2}}
Let $t\in [\pi/3,2\pi/3]$ be fixed. The fact that \(x\mapsto F(\Phi\circ \Psi(x),t)\)
is decreasing for $x\in (0,\infty)$ follows directly from the facts that $\Phi\circ \Psi(x)$ is increasing for $x\in (0,\infty)$, $x\mapsto F(x,t)$ is decreasing for $x\in [1,\infty)$ (Lemma \ref{lem:prop_of_F}) and $\Phi\circ \Psi ((0,\infty))\subset [1,\infty)$. Now, we want to prove that
 \( W(x,t):=F(x,t)+F(\Phi(x),t)+F(\Psi(x),t)   \)
    is decreasing for $x\in[\phi,\infty)$.  
  We will check that $\partial_x W(x,t)< 0$ for $x\in (\phi,\infty)$. A simple computation using 
 Lemma \ref{lem:prop_of_F}$(iii)$ shows that $\partial_x W(x,t)< 0$ is equivalent to 
 \begin{eqnarray*}
     \frac{(x^2-1)}{(1+x^2-2x\cos(t))^2} &>&  \frac{(2x+1)}{(2x^2+2x+1-2x(x+1)\cos(t))^2}\\
     &&  \quad \quad \quad \quad \quad \quad +\frac{ (x + 1) (3 x + 1)}{(5 x^2 + 4 x + 1- 2x( 2 x +1)\cos(t) )^2}.
 \end{eqnarray*}
We will prove the following two inequalities, which together imply the desired result:
 \begin{eqnarray*}
\frac{1}{2}\frac{(x^2-1)}{(1+x^2-2x\cos(t))^2} & > & \frac{(2x+1)}{(2x^2+2x+1-2x(x+1)\cos(t))^2},\\
\frac{1}{2} \frac{(x^2-1)}{(1+x^2-2x\cos(t))^2} & > & \frac{ (x + 1) (3 x + 1)}{(5 x^2 + 4 x + 1- 2x( 2 x +1)\cos(t) )^2}. 
 \end{eqnarray*}
These inequalities are equivalent to
 \begin{eqnarray}
\frac{1}{2}\frac{(x^2-1)}{(2x+1)} & >& \left(\frac{1+x^2-2x\cos(t)}{2x^2+2x+1-2x(x+1)\cos(t)}\right)^2, \label{ine1} \\
\frac{1}{2} \frac{(x^2-1)}{(x + 1) (3 x + 1) } & > & \left(\frac{1+x^2-2x\cos(t) }{5 x^2 + 4 x + 1- 2x( 2 x +1)\cos(t)}\right)^2.  \label{ine2}
 \end{eqnarray}
For fixed $x\in (\phi,\infty)$, the function
\(y\mapsto \frac{1+x^2-2xy}{2x^2+2x+1-2x(x+1)y}\)  for \(y\in [-1/2,1/2]\)
has derivative $\frac{2x^2(x^2-x-1)}{(2x^2+2x+1-2x(x+1)y)^2}> 0$, hence it is increasing. This implies that
\[\frac{1+x^2-2x\cos(t)}{2x^2+2x+1-2x(x+1)\cos(t)}\leq \frac{1+x^2-x}{2x^2+2x+1-x(x+1)}=\frac{1+x^2-x}{x^2+x+1}\]
since $\cos(t)\leq \frac{1}{2}$. Hence, in order to prove \eqref{ine1} we just need to check that
\[\frac{1}{2}\frac{(x^2-1)}{(2x+1)}  > \left(\frac{1+x^2-x}{x^2+x+1}\right)^2 \quad \text{for }x>\phi.\]
This is equivalent to
\[\tilde p(x):=x^6 - 2 x^5 + 8 x^4 - 8 x^3 - 2 x - 3>0 \quad \text{for }x>\phi.\]
Writing
\(\tilde p(x)= \left(x^4 \left(x - \frac{1}{2}\right) + \frac{29x^3}{4} + \frac{23 x^2}{8} + \frac{69 x}{16} + \frac{143}{32}\right)\left(x - \frac{3}{2}\right) + \frac{237}{64}\)
we see that $\tilde p(x)>0$ for $x>3/2$ and in particular $\tilde p(x)>0$ for $x\in (\phi,\infty)$. This proves \eqref{ine1}. For the inequality~\eqref{ine2} we first note that the function
\(y\mapsto \frac{1+x^2-2xy }{5 x^2 + 4 x + 1- 2x( 2 x +1)y}\) for \(y\in [-1/2,1/2]\)
has derivative $\frac{4x^2(x^2-2x-1)}{(5 x^2 + 4 x + 1- 2x( 2 x +1)y)^2}$. We then distinguish two cases: If~$x\in (\phi,1+\sqrt{2}]$ then $\frac{4x^2(x^2-2x-1)}{(5 x^2 + 4 x + 1- 2x( 2 x +1)y)^2}\leq 0$ and  \eqref{ine2} reduces to 
    \[\frac{1}{2} \frac{(x^2-1)}{(x + 1) (3 x + 1) }  > \left(\frac{1+x^2+x }{5 x^2 + 4 x + 1+x( 2 x +1)}\right)^2.\]
    This is equivalent to
    \[\tilde q(x):=43 x^5 + 7 x^4 - 53 x^3 - 47 x^2 - 19 x - 3> 0 \quad \text{for }x\in (\phi,1+\sqrt{2}].\]
    Writing
    \(\tilde q(x)=\left(43 x^4 + \frac{143 x^3}{2} + \frac{217 x^2}{4} + \frac{275 x}{8} + \frac{521}{16}\right)\left(x - \frac{3}{2}\right) + \frac{1467}{32}\)
    we conclude that $\tilde q(x)>0$ and \eqref{ine2} is proven. 
If $x\in [1+\sqrt{2},\infty)$ then $\frac{4x^2(x^2-2x-1)}{(5 x^2 + 4 x + 1- 2x( 2 x +1)y)^2}\geq 0$ and \eqref{ine2} reduces to 
    \[\frac{1}{2} \frac{(x^2-1)}{(x + 1) (3 x + 1) }  >  \left(\frac{1+x^2-x }{5 x^2 + 4 x + 1-x( 2 x +1)}\right)^2.\]
    This is equivalent to
    \[h(x):=3 x^5 + 19 x^4 - 17 x^3 - 3 x^2 - 7 x - 3> 0 \quad \text{for }x\in [1+\sqrt{2},\infty).\]
    Then, writing
    \(h(x) = (3 x^4 + 25 x^3 + 33 x^2 + 63 x + 119)(x - 2) + 235\)
    we deduce that $h(x)>0$ for $x>2$, and in particular $h(x)>0$ for $x\in [1+\sqrt{2},\infty)$. This proves \eqref{ine2}.

\subsection*{A.3.~Properties of convex and monotonic functions}


    Let $X\subseteq \R$ be a non-empty subset and $G:X\to \R$. Recall that $G$ is convex in $X$ if for all $a,b,c\in X$ with $a\leq b\leq c$ and $a\neq c$ we have
    \(G(b)\leq tG(a)+(1-t)G(c)\)
    where $t\in [0,1]$ is uniquely determined by $b=ta+(1-t)c$, namely $t=\frac{c-b}{c-a}$.

\begin{lemma}\label{lem:3-convexity_implies_convexity}
    Let $x_1<x_2<\ldots<x_m$ be real numbers with $m\geq 3$ and $X=\{x_i:1\leq i\leq m\}$. Assume that $G:X\to \mathbb{R}$ satisfies the following property: for all $1\leq i\leq m-2$ the function $G$ is convex in $\{x_i,x_{i+1},x_{i+2}\}$. Then, $G$ is convex in $X$.
\end{lemma}
\begin{proof}
    We use induction on $m\geq 3$. The case $m=3$ is trivial, so assume that the property is satisfied for $m\geq 3$. Let $y_1<\ldots<y_{m+1}$ be real numbers, $Y=\{y_i:1\leq i\leq m+1\}$ and $G:Y\to \mathbb{R}$ a function that is convex in $\{y_i,y_{i+1},y_{i+2}\}$ for all $1\leq i\leq m-1$. By induction, $G$ is convex in $A:=\{y_1,\ldots,y_m\}$ and also in $B:=\{y_2,\ldots,y_{m+1}\}$. We have to prove that $G$ is convex in $Y=A\cup B$, so let us take $1\leq i<j<k\leq m+1$ and consider the numbers $y_i<y_j<y_k$ in $Y$. Then, the inequality
    $$G(y_j)\leq tG(y_i)+(1-t)G(y_k) \quad  \text{with }t=\frac{y_k-y_j}{y_k-y_i}$$
    follows, in the case $2\leq i$, from the convexity of $G$ in $B$ since $\{y_i,y_j,y_k\}\subseteq B$. Similarly, in the case $k\leq m$, the result follows from the convexity of $G$ in $A$. Hence, we can assume $1=i<j<k=m+1$. Since by hypothesis $G$ is convex in  $\{y_{j-1},y_j,y_{j+1}\}$, we have
    \begin{equation}\label{eq:3-convexity-consecutive}
    G(y_j)\leq t_0G(y_{j-1})+(1-t_0)G(y_{j+1}) \quad  \text{with } t_0=\frac{y_{j+1}-y_j}{y_{j+1}-y_{j-1}}.        
    \end{equation}
    Now, we use convexity of $G$ in $\{y_1,y_{j-1},y_j\}\subseteq A$ (which in the case $j=2$ is trivially true\footnote{Indeed, if $j=2$ then $\{y_1,y_{j-1},y_j\}=\{y_1,y_1,y_j\}$ and since $y_1=1\cdot y_1+0\cdot y_j$ the convexity on this triple becomes $G(y_1)\leq 1\cdot G(y_1)+0\cdot G(y_j)$, which holds trivially.}) and the convexity in $\{y_{j},y_{j+1},y_{n+1}\}\subseteq B$ (trivially true if $j=n$) to get
            \begin{align}
         G(y_{j-1})\leq \alpha G(y_{1})+(1-\alpha)G(y_{j}) \quad  \text{with }\alpha=\frac{y_{j}-y_{j-1}}{y_{j}-y_{1}}, \label{eq:3-convexity-in-A} \\
         G(y_{j+1})\leq \beta G(y_{j})+(1-\beta)G(y_{n+1}) \quad  \text{with }\beta=\frac{y_{n+1}-y_{j+1}}{y_{n+1}-y_{j}}. \label{eq:3-convexity-in-B}
        \end{align}
        From \eqref{eq:3-convexity-consecutive}, \eqref{eq:3-convexity-in-A} and \eqref{eq:3-convexity-in-B} we obtain
\[G(y_j) \leq t_0\alpha G(y_{1})+(t_0(1-\alpha)+(1-t_0)\beta)G(y_j)+(1-t_0)(1-\beta)G(y_{n+1}). \]
Observe that $(t_0(1-\alpha)+(1-t_0)\beta)\in [0,1)$ since $t_0\in (0,1)$, $\alpha\in (0,1]$ and $\beta\in [0,1)$. Hence, we get
\begin{equation}\label{eq:3-convexity-extremes}
G(y_j) \leq \frac{t_0\alpha}{1-(t_0(1-\alpha)+(1-t_0)\beta)} G(y_{1})+\frac{(1-t_0)(1-\beta)}{1-(t_0(1-\alpha)+(1-t_0)\beta)}G(y_{n+1}).
\end{equation}
A straightforward computation shows that \eqref{eq:3-convexity-extremes} is exactly the convexity of $G$ in $\{y_1,y_j,y_{n+1}\}$. This completes the proof of the lemma.
\end{proof}

\begin{lemma}\label{lem:conv_and_global_min_implies_decreasing}
   Let $a<b$ be real numbers and $X\subseteq [a,b]$ containing $a$ and $b$. If $G:X\to \mathbb{R}$ is a  convex function satisfying $G(x)\geq G(b)$ for all $x\in X$, then $G$ is non-increasing. Moreover, if $G(x)> G(b)$ for all $x\in X\setminus \{b\}$ then $G$ is decreasing. 
\end{lemma}
\begin{proof} This follows directly from the definition of convexity, hence we omit the details.
    \end{proof}

\begin{lemma}\label{lem:extension_of_convex}
    Let~$a<b$ be rationals and~$G:[a,b]\cap \mathbb{Q}\to \mathbb{R}$ a convex function. Assume that $G(x)$ has a global maximum at $x=a$ and a global minimum at $x=b$. Then $G$ extends uniquely to a function $\overline{G}:[a,b]\to \mathbb{R}$ which is continuous in $(a,b]$ and both convex and non-increasing in $[a,b]$. Moreover, if $G(x_n)\to G(a)$ for at least one sequence $(x_n)_n$ in $(a,b]\cap \mathbb{Q}$ with $x_n\to a$, then $\overline{G}$ is continuous in $[a,b]$.
\end{lemma}

\begin{remark}
    The function $\overline{G}:[-1,0]$ given by $\overline{G}(x)=x^2$ if $x>-1$ and $\overline{G}(-1)=2$ is convex in $[-1,0]$, has global maximum at $x=-1$ and global minimum at $x=0$, but it is only continuous in $(-1,0]$. This shows that the first conditions in Lemma \ref{lem:extension_of_convex} are not sufficient to ensure continuity in the whole interval $[a,b]$.
\end{remark}

 \begin{proof}[Proof of Lemma \ref{lem:extension_of_convex}]
It is well known that the convexity of $G:[a,b]\cap \mathbb{Q}\to \R$ implies that $G$ is locally Lipschitz in $(a,b)\cap \mathbb{Q}$ (see, e.g., \cite[Exercise 17.37]{HS65}). It follows that $G|_{(a,b)\cap \mathbb{Q}}$ has a unique continuous extension $\overline{G}$ to $(a,b)$. It is defined, as usual, by
\[\overline{G}(t):=\lim_{n\to \infty} G(t_n) \quad \text{ for }t\in (a,b) \text{ and }t_n\in(a,b)\cap \mathbb{Q} \text{ with } t_n \to t.\]
We extend $\overline{G}$ to $[a,b]$ by putting $\overline{G}(x):=G(x)$ for $x=a,b$. It is easy to check that $\overline{G}:[a,b]\to \R$ is convex. By construction $\overline{G}(x)$ has a global maximum at $x=a$ and a global minimum at $x=b$. In particular, by Lemma \ref{lem:conv_and_global_min_implies_decreasing} we deduce that $G$ is non-increasing. We claim that $\overline{G}(x)$ is continuous at $x=b$. Since $x=b$ is a global minimum, we have 
\(\liminf\limits_{x\to b^-}\overline{G}(x)\geq \overline{G}(b).\)
On the other hand, for $x<b$ we write $x=ta+(1-t)b$ with $t=\frac{b-x}{b-a}$, and by convexity we get
$$\overline{G}(x)\leq \left(\frac{b-x}{b-a}\right) \overline{G}(a)+\left(\frac{x-a}{b-a}\right)\overline{G}(b).$$
Taking limits we obtain
\(\limsup\limits_{x\to b^-}\overline{G}(x)\leq \overline{G}(b).\)
This implies $\overline{G}(x)\to \overline{G}(b)$ as $x\to b^-$ and shows that $\overline{G}$ is continuous at $x=b$. Clearly, $\overline{G}$ is the unique extension of $G$ satisfying the first properties stated in the lemma. Finally, note that $x=a$ being a global maximum of $\overline{G}(x)$ implies
\(\limsup_{x\to a^+}\overline{G}(x)\leq \overline{G}(a).\)
If we further assume that $G(x_n)\to G(a)$ for at least one sequence $(x_n)_{n\geq 0}$ in $(a,b]\cap \mathbb{Q}$ with $x_n\to a$, then the same property holds for $\overline{G}$. Together with the monotonicity of $\overline{G}$ we conclude
\(\liminf_{x\to a^+}\overline{G}(x)= \overline{G}(a).\)
This shows that $\overline{G}(x)\to \overline{G}(a)$ as $x\to a^+$ and proves that $\overline{G}$ is continuous at $x=a$. This completes the proof of the lemma.
\end{proof}

\bibliographystyle{alpha}

\end{document}